\documentclass[a4paper,11pt]{amsart}
\usepackage{amssymb,amsfonts,amsxtra,amscd,mathrsfs}
\usepackage[all]{xy}
\usepackage{fullpage}
\usepackage{graphicx}
\theoremstyle{plain}
\newtheorem{theorem}{Theorem}[section]
\newtheorem{lemma}[theorem]{Lemma}
\newtheorem{cor}[theorem]{Corollary}
\newtheorem{prop}[theorem]{Proposition}
\theoremstyle{definition}
\newtheorem{defi}[theorem]{Definition}

\theoremstyle{remark}
\newtheorem{rem}[theorem]{Remark}
\numberwithin{equation}{section}
\newcommand{\ai}{\ensuremath{A_\infty}}
\newcommand{\dmcmp}{\ensuremath{\overline{\mathcal{M}}_{g,n}}}
\newcommand{\kcmp}{\ensuremath{\mathcal{K}\overline{\mathcal{M}}_{g,n}}}
\newcommand{\lcmp}{\ensuremath{\mathcal{L}\left[\dmcmp\times\Delta_n\right]}}
\newcommand{\gcat}{\ensuremath{\Gamma_{\mathrm{Iso}}((g,n))}}
\newcommand{\OTFT}{\ensuremath{\mathcal{OTFT}}}
\newcommand{\mcs}[1]{\ensuremath{\mathcal{MC}(#1)}}
\newcommand{\mcm}[1]{\ensuremath{\widetilde{\mathcal{MC}}(#1)}}
\newcommand{\gf}{\ensuremath{\mathbb{R}}}
\newcommand{\dilim}[2]{\ensuremath{\varinjlim_{#1} #2}}
\newcommand{\innprod}{\ensuremath{\langle -,- \rangle}}
\newcommand{\cotimes}{\ensuremath{\hat{\otimes}}}
\newcommand{\noproof}{\begin{flushright} \ensuremath{\square} \end{flushright}}
\DeclareMathOperator{\id}{id}
\DeclareMathOperator{\ad}{ad}
\DeclareMathOperator{\tr}{tr}
\DeclareMathOperator{\ch}{ch}
\DeclareMathOperator{\colim}{colim}
\DeclareMathOperator{\im}{Im}

\begin{document}
\title{Classes on the moduli space of Riemann surfaces through a noncommutative Batalin-Vilkovisky formalism}
\author{Alastair Hamilton}
\address{Texas Tech University, Department of Mathematics and Statistics, Lubbock TX 79407-1042. USA.}
\email{hamilton@math.uconn.edu}
\begin{abstract}
Using the machinery of the Batalin-Vilkovisky formalism, we construct cohomology classes on compactifications of the moduli space of Riemann surfaces from the data of a contractible differential graded Frobenius algebra. We describe how evaluating these cohomology classes upon a well-known construction producing homology classes in the moduli space can be expressed in terms of the Feynman diagram expansion of some functional integral. By computing these integrals for specific examples, we are able to demonstrate that this construction produces families of nontrivial classes.
\end{abstract}
\maketitle

\section{Introduction}

In this paper we describe a construction that produces cohomology classes on a compactification of the moduli space of Riemann surfaces from the initial data of a contractible differential graded Frobenius algebra. This is achieved by adapting the standard methodology of the Batalin-Vilkovisky formalism to a noncommutative geometric context. Much work has been done on placing the Batalin-Vilkovisky formalism within the framework of noncommutative geometry by Barannikov, see e.g. \cite{baran}. The construction described in this paper originates from the ideas of Kontsevich detailed in \cite{kontfeyn}. In this paper, he described a construction producing cohomology classes on moduli spaces of Riemann surfaces from the same initial data. It was discovered by Chuang-Lazarev in \cite{dualfeyn} that Kontsevich's construction naturally produces classes in a certain compactification of the moduli space, closely related to the compactification that he introduced in \cite{kontairy} in his proof of Witten's conjectures, and they gave in this paper a careful treatment of this construction which extended Kontsevich's original ideas to the framework of modular operads, hence leading to compactifications of the moduli space of Riemann surfaces.

In Kontsevich's original paper \cite{kontfeyn} he described a second construction producing \emph{homology} classes in moduli spaces of Riemann surfaces. It was his proposal that if one were to evaluate one construction upon the other, the answer could be expressed as a functional integral over a finite-dimensional space of fields, using the framework of the Batalin-Vilkovisky formalism. He proposed that by computing the perturbative expansion of such an integral, one could determine whether the corresponding homology and cohomology classes were the nontrivial.

One issue with Kontsevich's original proposal was that this second construction produced homology classes in the one point compactification of the moduli space. However, the construction producing cohomology classes produced these classes in a larger compactification of the moduli space. Hence, it was necessary to describe a way in which these homology classes could be lifted to this larger compactification. This question was addressed and solved in \cite{hamqme} were an obstruction theory controlling the process was laid out. Barannikov's paper \cite{baran} also bears on this point. The key ingredient in the development of this obstruction theory was the extension of a result by Kontsevich \cite{kontsympgeom} showing that the homology of the one point compactification of the moduli space could be recovered as the Chevalley-Eilenberg homology of a certain Lie algebra -- a noncommutative analogue of the Poisson algebra of Hamiltonian vector fields on a symplectic manifold. In \cite{hamcompact} it was shown that the homology of this larger compactification of the moduli space could be recovered as the Chevalley-Eilenberg homology of a certain differential graded Lie algebra, itself an object built using noncommutative geometry.

In this paper, we describe a noncommutative version of the Batalin-Vilkovisky formalism built upon this differential graded Lie algebra. This is achieved relatively simply by relying upon the classical theorems of the Batalin-Vilkovisky formalism and transferring these theorems to our new context. We use this framework to describe a construction producing Chevalley-Eilenberg cocycles of this differential graded Lie algebra, and hence cohomology classes in compactifications of the moduli space by the main theorem of \cite{hamcompact}. This allows us to describe the evaluation of this cohomology class upon the homology class described in \cite{hamcompact} as the perturbative expansion of some functional integral over a finite-dimensional space of fields. The construction of this cohomology class should be equivalent to the one described by Chuang-Lazarev in \cite{dualfeyn} using the alternative perspective of modular operads, however the advantage of our construction is its suitability for performing calculations. We consider, for the first time since the idea was initially proposed by Kontsevich in \cite{kontfeyn}, concrete examples of this construction and demonstrate through such computations that the cohomology classes that are produced are nontrivial.

The construction of these cohomology classes depend upon a contractible differential graded Frobenius algebra for their initial data. It is well-known that Frobenius algebras and (open) two-dimensional topological field theories are the same thing, and we will see that it is these topological field theory axioms that will play the crucial role in ensuring that the constructions that we perform actually produce cocycles and not just cochains. Of course, matrix algebras provide obvious examples of Frobenius algebras, and we will show that these examples lead to matrix integrals similar in form, though not identical, to those considered by Kontsevich in \cite{kontairy}.

The layout of the paper is as follows. In Section \ref{sec_ncgeom} we recall the work of Kontsevich \cite{kontsympgeom} on noncommutative symplectic geometry, as well as introducing the differential graded Lie algebra that is the subject of the main theorem of \cite{hamcompact} relating its Chevalley-Eilenberg homology to that of compactifications of the moduli space of Riemann surfaces. This theorem is recalled at the end of the section. In Section \ref{sec_otft} we recall the definition of an open topological field theory as phrased in the language of modular operads by \cite{dualfeyn} and use this to write a list of identities that will be used in our construction of a cocycle later in the paper. In Section \ref{sec_cbvform} we recall standard theorems from the Batalin-Vilkovisky formalism. In Section \ref{sec_ncbv} we adapt these basic theorems to the framework of noncommutative geometry and use them to define our main construction of a cohomology class on compactifications of the moduli space. In Section \ref{sec_char} we recall basic material about characteristic classes. In Section \ref{sec_pairing} we describe how the evaluation of the cohomology class produced by the construction defined in Section \ref{sec_ncbv} upon the homology class produced by the construction defined in Section \ref{sec_char} can be described in terms of a functional integral over a finite-dimensional space of fields. Finally, in Section \ref{sec_compute} we compute some examples and demonstrate that the classes produced by these constructions are nontrivial. We then describe how these constructions are related to matrix integrals.

\subsection{Notation and conventions}

$\Delta_n$ will denote the closed simplex of dimension $n-1$:
\[ \Delta_n:=\{(x_1,\ldots,x_n)\in [0,1]^{n+1}: \sum_i x_i =1 \}. \]
$\Delta_n^\circ$ will denote the open simplex of dimension $n-1$, which is just the interior of the closed simplex. Given a topological space $X$, the space $X^\infty$ will denote its one point compactification.

Given a $\mathbb{Z}/2\mathbb{Z}$-graded space $V$, the parity reversion $\Pi V$ is the $\mathbb{Z}/2\mathbb{Z}$-graded space
\[ \Pi V_0:= V_1 \qquad \Pi V_1:= V_0. \]
Throughout the paper we work over the ground field $\gf$. We will denote the algebra of Laurent polynomials in $h$ by $\gf[h,h^{-1}]$.

A \emph{symplectic vector space} is a $\mathbb{Z}/2\mathbb{Z}$-graded finite-dimensional vector space $V$ with a nondegenerate \emph{skew-symmetric} bilinear form
\[ \innprod: V\otimes V\to\mathbb{R}. \]
This bilinear form is required to be homogeneous of either even or odd degree. Up to isomorphism, there is only one symplectic vector space of fixed dimension with an \emph{odd} symplectic form. This is the vector space $W^1_{n|n}$ with $n$ even coordinates $x_i\in[W^1_{n|n}]^*$ and $n$ odd coordinates $\xi_i\in[W^1_{n|n}]^*$. The bilinear form is given by
\[ \innprod:=\sum_i x_i\otimes\xi_i - \xi_i\otimes x_i. \]

Over $\mathbb{C}$, there is only one symplectic vector space $W^0_{2n|m}$ of fixed dimension $2n|m$ with an \emph{even} symplectic form. It has $2n$ even coordinates $p_i,q_i\in [W^0_{2n|m}]^*$ and $m$ odd coordinates $\xi_i\in [W^0_{2n|m}]^*$. The bilinear form is given by
\[ \innprod:=\sum_i [p_i\otimes q_i - q_i\otimes p_i] + \sum_i x_i\otimes x_i. \]

Given a nondegenerate bilinear form $\innprod$ on a vector space $V$, there are isomorphisms
\[ D_l:V\to V^* \quad\text{and}\quad D_r:V\to V^* \]
given by the formulae
\[ D_l(u):=[x\mapsto\langle u,x \rangle] \quad\text{and}\quad D_r(u):=[x\mapsto\langle x,u \rangle]. \]
From this we can define the \emph{inverse bilinear form} $\innprod^{-1}$ from the following commutative diagram:
\[ \xymatrix{ & \gf \\ V\otimes V \ar[ur]^{\innprod} \ar[rr]^{D_l \otimes D_r} && V^*\otimes V^* \ar[ul]_{\innprod^{-1}}} \]

There is a one-to-one correspondence between quadratic superfunctions $\sigma\in ([V^*]^{\otimes 2})_{\mathbb{S}_2}$ and \emph{symmetric} bilinear forms $\innprod\in ([V^*]^{\otimes 2})^{\mathbb{S}_2}$ defined by the following map
\begin{displaymath}
\begin{array}{ccc}
\sigma\in ([V^*]^{\otimes 2})_{\mathbb{S}_2} & \to & \innprod\in ([V^*]^{\otimes 2})^{\mathbb{S}_2} \\
x\otimes y & \mapsto & x\otimes y + (-1)^{xy}y\otimes x
\end{array}
\end{displaymath}
For any even symmetric nondegenerate bilinear form $\innprod$, we can always find even coordinates $x_1,\ldots,x_n\in V^*$ and odd coordinates $\xi_1,\ldots,\xi_{2m}\in V^*$ such that $\innprod$ has the form
\[ \innprod=\sum_{i=1}^k x_i\otimes x_i - \sum_{i=k+1}^n x_i\otimes x_i + \sum_{i=1}^m \xi_{2i-1}\otimes\xi_{2i} - \xi_{2i}\otimes\xi_{2i-1}.\]
The quadratic function $\sigma$ corresponding to this form is
\[ \sigma=\frac{1}{2}\sum_{i=1}^k x_i^2 - \frac{1}{2}\sum_{i=k+1}^n x_i^2 + \sum_{i=1}^m \xi_{2i-1}\xi_{2i}. \]

By a differential graded Lie algebra we will mean a vector space $\mathfrak{g}$ equipped with an \emph{odd} bracket
\[ \{-,-\}:\mathfrak{g}\otimes\mathfrak{g}\to\mathfrak{g} \]
and a compatible differential. This means that the bracket $[-,-]$ defined by the equation
\[ \Pi\circ [-,-] = \{-,-\}\circ (\Pi\otimes\Pi) \]
satisfies the Jacobi identity and the Leibniz rule with respect to the differential. Since we deal with the Batalin-Vilkovisky formalism in this paper, we will predominantly deal with odd brackets; however we will also occasionally deal with conventional even brackets as well.

\section{Noncommutative geometry} \label{sec_ncgeom}

In this section we recall the relevant material from \cite{kontsympgeom} regarding the framework of noncommutative symplectic geometry provided by Kontsevich. Throughout the paper, we by and large treat the case of an odd symplectic form, indicating in subsequent remarks any slight differences that arise in the case of an even symplectic form. Since we ultimately intend to perform calculations in this setup, we treat sign rules carefully. We use this framework together with algebraic structures considered in \cite{movshev} and \cite{sched} to define a differential graded Lie algebra depending on two parameters. In \cite{hamcompact} it was shown that the Chevalley-Eilenberg homology of this differential graded Lie algebra coincides with the homology of a certain compactification of the moduli space of Riemann surfaces introduced by Kontsevich in \cite{kontairy}. After defining Chevalley-Eilenberg homology, we recall this theorem at the end of the section. In subsequent sections we will rely on this description of the homology of the moduli space to provide the necessary link between moduli spaces of Riemann surfaces and the framework of the Batalin-Vilkovisky formalism, which will enable us to give a construction of classes in the moduli space using the standard theory that appears naturally in the latter framework.

\subsection{Lie algebras of vector fields}

Let us start by recalling the noncommutative symplectic geometry of Kontsevich. The following object plays the role in noncommutative geometry of the Poisson algebra of Hamiltonian vector fields on a symplectic vector space.

\begin{defi}
Let $(V,\innprod)$ be a symplectic vector space with an \emph{odd} symplectic form. In \cite{kontsympgeom} Kontsevich introduced the structure of a Lie algebra on the space of noncommutative functions
\[\mathfrak{h}[V]:=\bigoplus_{n=0}^\infty [(V^*)^{\otimes n}]_{\mathbb{Z}/n\mathbb{Z}}.\]
The Lie bracket $\{-,-\}$ on $\mathfrak{h}[V]$ is defined by
\begin{equation} \label{eqn_bracket}
\{a_1\cdots a_n,b_1\cdots b_m\} = \sum_{i=1}^n\sum_{j=1}^m (-1)^p \langle a_i,b_j \rangle^{-1} (z_{n-1}^{i-1}\cdot [a_1\cdots \hat{a_i} \cdots a_n]) (z_{m-1}^{j-1}\cdot [b_1\cdots \hat{b_j} \cdots b_m]),
\end{equation}
where $a_1,\ldots,a_n,b_1,\ldots,b_m\in V^*$, $z_n$ is the $n$-cycle $(n \, n-1\, \ldots 2\, 1)$ and
\[ p:=|a_i|(|a_1|+\cdots+|a_{i-1}|) +|b_j|(|a_1|+\cdots+|a_n|+|b_1|+\cdots+|b_{j-1}|). \]
\end{defi}

\begin{rem}
When the symplectic form $\innprod$ is \emph{even}, the space $\mathfrak{h}[V]$ can be endowed with an \emph{even} Lie bracket using the same formula, the only difference being the sign
\[ p:=|a_i|(|a_{i+1}|+\cdots+|a_k|+|b_1|+\cdots+|b_{j-1}|). \]
\end{rem}

It was discovered by Movshev in \cite{movshev} that $\mathfrak{h}[V]$ can be endowed with the extra structure of an (involutive) Lie bialgebra.

\begin{defi}
The (odd) Lie cobracket on $\mathfrak{h}[V]$ is defined by the formula
\begin{equation} \label{eqn_cobracket}
\Delta(a_1\cdots a_n) = \frac{1}{2}\sum_{i<j}(-1)^p\langle a_i,a_j \rangle^{-1}[1+(1 \, 2)]\cdot[ (a_{i+1}\cdots a_{j-1}) \otimes (a_{j+1}\cdots a_n a_1\cdots a_{i-1})],
\end{equation}
where $a_1,\ldots,a_n\in V^*$ and
\begin{displaymath}
\begin{split}
p:= & |a_i|(|a_1|+\cdots+|a_i|) + |a_j|(|a_1|+\cdots+|a_j|) \\
& + (|a_1|+\cdots+|a_{i-1}|)(|a_{i+1}+\cdots+|a_{j-1}|+|a_{j+1}|+\cdots+|a_n|)
\end{split}
\end{displaymath}
\end{defi}

\begin{rem}
Again, when the inner product $\innprod$ is even, the space $\mathfrak{h}[V]$ may be endowed with an even Lie cobracket using the same formula, the only difference being the sign
\[ p:= (|a_1|+\cdots+|a_{i-1}|)(|a_{i+1}|+\cdots+|a_{j-1}|+|a_{j+1}|+\cdots+|a_n|) + |a_j|(|a_{i+1}|+\cdots+|a_{j-1}|) \]
and the fact that $[1+(1 \, 2)]$ gets replaced with $[1-(1 \, 2)]$
\end{rem}

\begin{rem}
This Lie bialgebra $\mathfrak{h}[V]$ has a natural grading determined by the order of a cyclic polynomial. We will denote by
\[ \mathfrak{h}_{\geq k}[V]:=\bigoplus_{n=k}^\infty [(V^*)^{\otimes n}]_{\mathbb{Z}/n\mathbb{Z}} \]
the Lie subalgebra (when $k\geq 2$) of cyclic polynomials of order $\geq k$.
\end{rem}

From this Lie bialgebra $(\mathfrak{h}[V],\{-,-\},\Delta)$ we can construct a differential graded Lie algebra in a canonical way. First, we must introduce the definition of the Chevalley-Eilenberg complex of a Lie algebra.

\subsection{Chevalley-Eilenberg homology}

\begin{defi} \label{def_cecomp}
Let $\mathfrak{g}$ be a differential graded Lie algebra. The Chevalley-Eilenberg complex of $\mathfrak{g}$, denoted by $C_\bullet(\mathfrak{g})$ is the complex with underlying vector space
\[ C_\bullet(\mathfrak{g}):=S(\mathfrak{g})=\bigoplus_{i=0}^\infty [\mathfrak{g}^{\otimes i}]_{\mathbb{S}_i} \]
and with differential $\delta:C(\mathfrak{g})\to C(\mathfrak{g})$ defined by
\begin{equation} \label{eqn_cediff}
\begin{split}
\delta(g_1\cdots g_n):=&\sum_{1\leq i<j\leq n} (-1)^p\{g_i,g_j\}\cdot g_1\cdots \hat{g_i} \cdots \hat{g_j} \cdots g_n \\
& + \sum_{1\leq i\leq n} (-1)^q d(g_i)\cdot g_1\cdots \hat{g_i} \cdots  g_n;
\end{split}
\end{equation}
where $d$ is the differential on $\mathfrak{g}$ and
\begin{displaymath}
\begin{split}
p:= & |g_i|(|g_1|+\cdots+|g_{i-1}|) + |g_j|(|g_1|+\cdots+|g_{j-1}|)+|g_i||g_j| \\
q:= & |g_i|(|g_1|+\cdots+|g_{i-1}|)
\end{split}
\end{displaymath}
The homology of this complex is called the \emph{Chevalley-Eilenberg homology} of $\mathfrak{g}$ with trivial coefficients and is denoted by $H_\bullet(\mathfrak{g})$. The cohomology of this complex, that is the homology of the linearly dual complex, is called the \emph{Chevalley-Eilenberg cohomology} of $\mathfrak{g}$ with trivial coefficients and is denoted by $H^\bullet(\mathfrak{g})$.
\end{defi}

We will also need to define \emph{relative} Chevalley-Eilenberg cohomology as well.

\begin{defi}
Suppose that the differential graded Lie algebra $\mathfrak{g}$ contains $\mathfrak{q}$ as a differential graded Lie sub algebra. Consider the quotient
\[ C_\bullet(\mathfrak{g},\mathfrak{q}):=S(\mathfrak{g}/\mathfrak{q})_{\mathfrak{q}} \]
of $S(\mathfrak{g}/\mathfrak{q})$ by the adjoint action of $\mathfrak{q}$. The differential \eqref{eqn_cediff} descends to this quotient and forms the relative Chevalley-Eilenberg complex of $\mathfrak{g}$ modulo $\mathfrak{q}$. The homology of this complex is referred to as the \emph{relative Chevalley-Eilenberg homology} of $\mathfrak{g}$ modulo $\mathfrak{q}$ and is denoted by $H_\bullet(\mathfrak{g},\mathfrak{q})$. The cohomology of this complex is called \emph{relative Chevalley-Eilenberg cohomology} and denoted by $H^\bullet(\mathfrak{g},\mathfrak{q})$.
\end{defi}

\subsection{Differential graded Lie algebra} \label{sec_dgla}

Now we may describe how to construct our two parameter family of differential graded Lie algebras. We may define a grading on $S(\mathfrak{g})$; we will denote by
\[ S_{\geq k}(\mathfrak{g}):=\bigoplus_{n=k}^\infty [\mathfrak{g}^{\otimes n}]_{\mathbb{S}_n} \]
the subspace of tensors of order $\geq k$. If we take a symplectic vector space $(V,\innprod)$ whose symplectic form is odd and look at the Chevalley-Eilenberg complex $C_\bullet(\mathfrak{h}[V])$ of the associated Lie algebra $\mathfrak{h}[V]$, we may extend the cobracket $\Delta$ on $\mathfrak{h}[V]$ to another differential on $C_\bullet(\mathfrak{h}[V])$ using the Leibniz rule
\[ \Delta(h_1\cdots h_n):=\sum_{i=1}^n (-1)^p\Delta(h_i)\cdot h_1\cdots \hat{h_i} \cdots h_n \]
where $p:=|h_i|(|h_1|+\cdots+|h_{i-1}|)$. Furthermore, we may extend the Lie bracket on $\mathfrak{h}[V]$ to one on $C_\bullet(\mathfrak{h}[V])$ by again using the Leibniz rule
\begin{equation} \label{eqn_CEbracket}
\{g_1\cdots g_n,h_1\cdots h_m\}=\sum_{i=1}^n\sum_{j=1}^m (-1)^p\{g_i,h_j\}\cdot g_1\cdots \hat{g_i} \cdots g_n\cdot h_1\cdots \hat{h_j} \cdots h_m
\end{equation}
where
\[ p:=|g_i|(|g_1|+\cdots+|g_{i-1}|) + |h_j|(|g_1|+\cdots+|g_n|+|h_1|+\cdots+|h_{j-1}|) + |g_i||h_j| \]

This allows us to consider the following object. Define
\[ \mathfrak{l}:=\gf[\gamma]\otimes C_\bullet(\mathfrak{h}[V]) \]
and equip it with the bracket naturally arising from the bracket on $\mathfrak{h}[V]$ and the differential $d$ defined by the formula
\[ d:=\gamma\cdot\delta+\Delta. \]

\begin{lemma}
$\mathfrak{l}$ is a differential graded Lie algebra.
\end{lemma}

\begin{proof}
This is a standard construction, see e.g. \cite{hamcompact} for details.
\end{proof}

To get our desired differential graded Lie algebras, we will modify $\mathfrak{l}$ in an appropriate way. First note that
\[ \mathfrak{h}[V]=\gf\oplus\mathfrak{h}_{\geq 1}[V]. \]
By identifying the symmetric algebra on the ground field $\gf$ with the free polynomial algebra in one variable $\nu$ we can write
\[ \mathfrak{l}=\gf[\gamma,\nu]\otimes S(\mathfrak{h}_{\geq 1}[V]). \]
Since the summand $\gf[\gamma,\nu]$ is a differential graded ideal of $\mathfrak{l}$, the subspace
\[ \mathfrak{l}':=\gf[\gamma,\nu]\otimes S_{\geq 1}(\mathfrak{h}_{\geq 1}[V]) \]
inherits the structure of a differential graded Lie algebra. From this differential graded Lie algebra, we pick the differential graded Lie subalgebra $\Lambda_{\gamma,\nu}[V]$ obtained by throwing out the summand corresponding to $V^*$; that is to say that $\mathfrak{l}'$ has a grading in which the deformation parameter $\gamma$ has order 2, the parameter $\nu$ has order 1 and a cyclic word in $\mathfrak{h}[V]$ of length $i$ has order $i$, and we take $\Lambda_{\gamma,\nu}[V]$ to be the Lie subalgebra generated by terms of order $\geq 2$. We define $\Lambda_\gamma[V]$ to be the subspace of $\gf[\gamma]\otimes S_{\geq 1}(\mathfrak{h}_{\geq 1}[V])$ generated by terms of order $\geq 2$. It inherits the structure of a differential graded Lie algebra by stipulating that the natural projection $\Lambda_{\gamma,\nu}[V]\to\Lambda_{\gamma}[V]$ determined by setting the deformation parameter $\nu=0$ is a map of differential graded Lie algebras.

\begin{rem}
If the symplectic form $\innprod$ is even, we proceed slightly differently. In this case $\mathfrak{h}[V]$ has an \emph{even} bracket, hence the parity reversion $\Pi\mathfrak{h}[V]$ has an odd bracket. We can repeat the preceding construction on $\Pi\mathfrak{h}[V]$ to form a differential graded Lie algebra
\[ \mathfrak{l}:=\gf[\gamma]\otimes S(\Pi\mathfrak{h}[V])=\gf[\gamma,\nu]\otimes S(\Pi\mathfrak{h}_{\geq 1}[V]) \]
where the deformation parameter $\nu$ is now \emph{odd}. The symmetric algebra $S(\Pi\mathfrak{h}_{\geq 1}[V])$ can be identified with the graded exterior algebra $\Lambda(\mathfrak{h}_{\geq 1}[V])$ in the usual way using the Koszul sign rule. The differential graded Lie algebras $\Lambda_{\gamma,\nu}[V]$ and $\Lambda_{\gamma}[V]$ are formed in the same way as before.
\end{rem}

\subsection{Relationship to moduli space}

We now describe the relationship between the differential graded Lie algebras constructed in the previous section and compactifications of the moduli space of Riemann surfaces. These compactifications were introduced by Kontsevich in \cite{kontairy} in his proof of Witten's conjectures and by Looijenga in a subsequent paper \cite{looi} which extended and clarified Kontsevich's work. In order to connect the material of the previous section with moduli spaces of Riemann surfaces, we will need to recall how to define a Hopf algebra structure on the stable Chevalley-Eilenberg homology of the differential graded Lie algebras that we have just constructed.

Given a symplectic vector space $(V,\innprod)$, we may consider the Lie subalgebra of $\mathfrak{h}_{\geq 2}[V]$ consisting of \emph{quadratic} superfunctions. This Lie subalgebra may be identified with the Lie algebra of endomorphisms of $V$ which annihilate the symplectic form $\innprod$. These Lie algebras have special names; when the symplectic form $\innprod$ is odd it is denoted by $\mathfrak{pe}[V]$, when $\innprod$ is even it is denoted by $\mathfrak{osp}[V]$. One can check that this Lie subalgebra is also a differential graded Lie subalgebra of both $\Lambda_{\gamma,\nu}[V]$ and $\Lambda_{\gamma}[V]$.

Now let us define stable versions of the objects that we have defined. Consider the canonical symplectic vector spaces $W^1_{n|n}$ and $W^0_{2n|m}$. We make the definitions:
\begin{displaymath}
\begin{split}
\Lambda_{\gamma,\nu}^1:= \dilim{n}{\Lambda_{\gamma,\nu}[W^1_{n|n}]} & \qquad \Lambda_{\gamma}^1:= \dilim{n}{\Lambda_{\gamma}[W^1_{n|n}]} \\
\mathfrak{h}^1:= \dilim{n}{\mathfrak{h}[W^1_{n|n}]} & \qquad \mathfrak{pe}:= \dilim{n}{\mathfrak{pe}[W^1_{n|n}]} \\
\Lambda_{\gamma,\nu}^0:= \dilim{n,m}{\Lambda_{\gamma,\nu}[W^0_{2n|m}]} & \qquad \Lambda_{\gamma}^0:= \dilim{n,m}{\Lambda_{\gamma}[W^0_{2n|m}]} \\
\mathfrak{h}^0:= \dilim{n,m}{\mathfrak{h}[W^0_{2n|m}]} & \qquad \mathfrak{osp}:= \dilim{n,m}{\mathfrak{osp}[W^0_{2n|m}]} \\
\end{split}
\end{displaymath}

Consider the relative mod $\mathfrak{pe}$ Chevalley-Eilenberg complex of the Lie algebra $\mathfrak{h}^1$. This complex has a natural commutative multiplication derived from the morphism of Lie algebras
\[ \mathfrak{h}[W^1_{n|n}] \oplus \mathfrak{h}[W^1_{m|m}] \to \mathfrak{h}[W^1_{n+m|n+m}]. \]
Combining this multiplication with the canonical coproduct on this complex endows the relative Chevalley-Eilenberg complex with the structure of a commutative cocommutative differential graded Hopf algebra. Likewise, the relative mod $\mathfrak{pe}$ Chevalley-Eilenberg complexes of $\Lambda_{\gamma,\nu}^1$ and $\Lambda_{\gamma}^1$ are similarly endowed with such a structure. The Milnor-Moore theorem \cite{mmthm} implies that these Hopf algebras are polynomial algebras in their primitive elements. Precisely the same remarks apply to the relative mod $\mathfrak{osp}$ Chevalley-Eilenberg complexes of $\Lambda_{\gamma,\nu}^0$, $\Lambda_{\gamma}^0$ and $\mathfrak{h}^0$.

The relative Chevalley-Eilenberg complexes of the differential graded Lie algebras that we have constructed provide an alternative description of complexes of cellular chains on certain compactifications of the moduli space of curves. These compactifications are constructed as quotients of the well-known Deligne-Mumford compactification $\dmcmp$.

There are two compactifications of interest. The first compactification, which we will refer to as the \emph{Kontsevich compactification}, is the quotient of the Deligne-Mumford compactification $\dmcmp$ that we obtain by forgetting the complex structure on those irreducible components which contain no marked points and remembering only their topological type. More precisely, we say that two stable curves in $\dmcmp$ are equivalent if, when we contract those irreducible components with no marked points and label the resulting nodal singularities by their (arithmetic) genus, there is a biholomorphic mapping between the resulting surfaces which preserves the labeling by the genus at the nodes; see e.g. \cite{mondello} or \cite{zvonkine} for details. The quotient of $\dmcmp$ by this equivalence relation will be called the Kontsevich compactification and will be denoted by $\kcmp$.

The second compactification will be referred to as the \emph{Looijenga compactification}. In this compactification, we consider \emph{decorated} stable curves in which each marked point is decorated by a nonnegative real number, called a \emph{perimeter}, which is allowed to vanish. We say that two decorated stable curves in $\dmcmp\times\Delta_n$ are equivalent if, when we contract those irreducible components which only contain points with vanishing perimeters and label the resulting nodal singularities by their genus and number of marked points, there is a biholomorphic mapping between the resulting surfaces preserving the labels at the nodes. The quotient of $\dmcmp\times\Delta_n$ by this equivalence relation will be called the Looijenga compactification and denoted by $\lcmp$. Note that both $\kcmp$ and $\lcmp$ have finite dimensional homology.

In what follows we will consider the moduli space with \emph{unlabeled} marked points, in which we take the quotient of the above spaces by the action of the symmetric group $\mathbb{S}_n$. The following theorem was proven in \cite{hamcompact}:

\begin{theorem} \label{thm_modspc}
Let us use the prefix $\mathcal{P}$ to denote primitive homology. There are isomorphisms in homology:
\begin{displaymath}
\begin{split}
\mathcal{P}H_\bullet(\Lambda_{\gamma,\nu}^1,\mathfrak{pe}) & \cong \bigoplus_{\chi=-1}^{-\infty} \left[\bigoplus_{\begin{subarray}{c} g\geq 0, \ n\geq 1: \\ \chi=2-2g-n  \end{subarray}} H_\bullet(\lcmp/\mathbb{S}_n) \right] \\
\mathcal{P}H_\bullet(\Lambda_{\gamma}^1,\mathfrak{pe}) & \cong \bigoplus_{\chi=-1}^{-\infty} \left[\bigoplus_{\begin{subarray}{c} g\geq 0, \ n\geq 1: \\ \chi=2-2g-n  \end{subarray}} H_\bullet([\kcmp\times\Delta_n^\circ]^\infty/\mathbb{S}_n) \right]
\end{split}
\end{displaymath}
\end{theorem}
\noproof

\begin{rem}
There is a similar theorem for the differential graded Lie algebras defined from an even symplectic form. If we replace the differential graded Lie algebras $\Lambda_{\gamma,\nu}^1$ and $\Lambda_{\gamma}^1$ with the differential graded Lie algebras $\Lambda_{\gamma,\nu}^0$ and $\Lambda_{\gamma}^0$ and replace $\mathfrak{pe}$ with $\mathfrak{osp}$, then the above theorem still holds, except that on the right we must consider the homology of the moduli space \emph{with twisted coefficients} (cf. Theorem 4.1 of \cite{lazvor}).
\end{rem}

\section{Topological field theories} \label{sec_otft}

In this section we recall a reformulation, provided by \cite{dualfeyn} using the language of modular operads, of the (well-known) definition of a topological field theory as originally formulated by Atiyah-Segal and others. This will be relevant to our construction of a cocycle on the moduli space. The natural axioms of an open topological field theory will be translated into a set of identities for a family of tensors. These identities, it will later be shown, will imply that the constructions we carry out in the subsequent parts of the paper actually produce \emph{cocycles} on the moduli space.

\subsection{Modular operads}

We begin by recalling the notion of modular operad due to Getzler-Kapranov \cite{getkap}, starting with the definition of a stable $\mathbb{S}$-module.

\begin{defi}
A stable $\mathbb{S}$-module is a collection of vector spaces, or more generally, chain complexes
\[\mathcal{V}((g,n))\]
defined for $g,n\geq 0$ such that $2g+n-2>0$ and equipped with an action of $\mathbb{S}_n$ on each $\mathcal{V}((g,n))$. Morphisms of $\mathbb{S}$-modules are just equivariant maps respecting the grading by $g$ and $n$.

Given a finite set $I$ we define
\[ \mathcal{V}((g,I)):=\left[\bigoplus_{\begin{subarray}{c} \text{bijections} \\ \{1,\ldots,n\}\to I \end{subarray}} \mathcal{V}((g,n))\right]_{\mathbb{S}_n} \]
where $\mathbb{S}_n$ acts on $\mathcal{V}((g,n))$ and by permuting summands.
\end{defi}

In order to define the notion of modular operad we will need to introduce graphs.

\begin{defi}
A stable graph (with legs) is a set $G$, called the set of \emph{half-edges}, together with the following data:
\begin{enumerate}
\item
A disjoint collection of pairs of elements of $G$, denoted by $E(G)$, called the set of \emph{edges} of $G$. Those half-edges which are not part of an edge are called the \emph{legs} of $G$.
\item
A partition of $G$, denoted by $V(G)$, called the set of \emph{vertices} of $G$. We will refer to the cardinality $n(v)$ of a vertex $v\in V(G)$ as the \emph{valency} of $v$.
\item
For every vertex $v\in V(G)$, a nonnegative integer $g(v)$ called the \emph{genus} of $v$. We impose the condition that $2g(v)+n(v)-2$ must be positive at every vertex $v\in V(G)$.
\end{enumerate}
In addition, a stable graph $G$ must be \emph{connected}.
\end{defi}

We define the \emph{genus} of a stable graph $G$ by the formula
\[ g(G):=b_1(G)+\sum_{v\in V(G)} g(v), \]
where $b_1(G)$ refers to the first Betti-number of the geometric realization of $G$. We define the category $\gcat$ to be the category whose objects are stable graphs of genus $g$ with $n$ legs labeled from $1$ to $n$ and whose morphisms are isomorphisms of stable graphs preserving the labeling of the legs.

For any stable graph $G$ and $\mathbb{S}$-module $\mathcal{V}((g,n))$, we define
\[ \mathcal{V}((G)):= \bigotimes_{v\in V(G)}\mathcal{V}((g(v),v)). \]
This allows us to define an endofunctor $\mathbb{M}$ on stable $\mathbb{S}$-modules by
\[ \mathbb{M}\mathcal{V}((g,n)):=\underset{G\in\text{Iso}\Gamma((g,n))}{\colim}\mathcal{V}((G)). \]
From this we can construct a triple $(\mathbb{M},\mu,\eta)$ \cite{getkap}. The natural transformation $\mu:\mathbb{M}\mathbb{M}\to\mathbb{M}$ is given by gluing the legs of the stable graphs, located at the vertices of some parent stable graph of which they are all subgraphs, along the edges of that parent stable graph. The natural transformation $\eta:\id\to\mathbb{M}$ is the map which associates to a $\mathbb{S}$-module $\mathcal{V}((g,n))$, the corolla whose single vertex is decorated by that $\mathbb{S}$-module.

\begin{defi} \label{def_modop}
A modular operad is an algebra over the triple $(\mathbb{M},\mu,\eta)$. A morphism of modular operads is just a morphism of such algebras. These maps are required to commute with the differentials.
\end{defi}

A fundamental example of a modular operad is provided by the \emph{endomorphism modular operad}.

\begin{defi}
Let $V$ be a finite-dimensional complex with a symmetric, even inner product $\innprod$ such that
\[ \langle d(x),y \rangle + (-1)^x\langle x,d(y) \rangle = 0. \]
The endomorphism modular operad of $V$, denoted by $\mathcal{E}[V]$, is the modular operad whose underlying $\mathbb{S}$-module is $\mathcal{E}[V]((g,n)):=(V^*)^{\otimes n}$. The structure map
\[ \mathbb{M}\mathcal{E}[V]\to\mathcal{E}[V] \]
is defined by contracting the tensors in $\mathcal{E}[V]$ along the edges of the graph using the inverse inner product $\innprod^{-1}$.
\end{defi}

This allows the following fundamental definition.

\begin{defi}
An \emph{algebra} over a modular operad $\mathcal{A}$ is a vector space $V$ together with a morphism of modular operads $\mathcal{A}\to\mathcal{E}[V]$.
\end{defi}

\subsection{Topological field theory}

We wish to cast the notion of a topological field theory within this framework (cf. \cite{dualfeyn}).

\begin{defi}
Given integers $\lambda\geq 0$ and $\nu, n \geq 1$, let $M_{\lambda,\nu,n}$ denote the category of connected compact oriented topological surfaces of genus $\lambda$ with $\nu$ boundary components and $n$ labeled intervals embedded in the boundary. That is an object in $M_{\lambda,\nu,n}$ is a connected compact oriented surface $S$ of genus $\lambda$ with $\nu$ boundary components, together with the data of $n$ orientation preserving embeddings $f_i:[0,1]\to\partial S$ for $1\leq i\leq n$.

\begin{figure}[htp]
\centering
\includegraphics{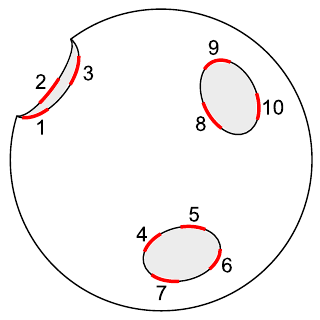}
\caption{Objects in $M_{\lambda,\nu,n}$ are surfaces with parameterized labeled intervals embedded in the boundary.}
\end{figure}

A morphism in $M_{\lambda,\nu,n}$ is just a morphism of topological spaces preserving the orientation and the labeled embedded intervals. We denote the set of isomorphism classes\footnote{This makes sense since there are only a finite number of such isomorphism classes.} by $[M_{\lambda,\nu,n}]$.
\end{defi}

\begin{defi}
The modular operad $\OTFT$ is defined as follows: its underlying $\mathbb{S}$-module is
\[ \OTFT((g,n)):=\bigoplus_{\begin{subarray}{c} \lambda\geq 0, \ \nu\geq 1: \\ 2\lambda+\nu-1=g \\  \end{subarray}} \gf[M_{\lambda,\nu,n}], \]
where $\mathbb{S}_n$ acts by relabeling the intervals embedded in the boundary. The structure map
\[ \mathbb{M} \OTFT \to \OTFT \]
of the modular operad is given by gluing surfaces along the embedded intervals on the boundary using the structure of the graph.
\end{defi}

\begin{rem}
Note, the genus $g$ in the modular operad $\OTFT$ does not correspond to the genus $\lambda$ of the surface, but to the quantity $2\lambda+\nu-1$, which is the first Betti number of the surface.
\end{rem}

An algebra over the modular operad $\OTFT$ is called an \emph{open topological field theory}. It is just a way to assign multilinear operations $V^{\otimes n}\to\gf$ to surfaces in a way that depends only on the topology of that surface and that is compatible with the various possible ways in which such topological surfaces may be glued together. This just amounts to a reformulation of the classic axioms first described by Atiyah-Segal et. al. in terms of modular operads. In this context, there is a classic theorem which says that an open topological field theory is simply a differential graded Frobenius algebra. One may consult \cite{dualfeyn} for a proof in this context.

\begin{theorem}
The datum of an open topological field theory consists of nothing more than a differential graded Frobenius algebra.
\end{theorem}
\noproof

We will explain the content of this theorem by describing precisely how a Frobenius algebra gives rise to an open topological field theory. Given any differential graded Frobenius algebra $A$, we can associate a collection of tensors
\begin{equation} \label{eqn_tensor}
\alpha_n^{g,b} \in T(A^*)^{\otimes n}
\end{equation}
for every $n, g, b\geq 0$. The components of this tensor
\[ \alpha_n^{g,b}|_{[k_1,\ldots,k_n]}:A^{\otimes k_1}\otimes\ldots\otimes A^{\otimes k_n}\to\gf \]
are defined as follows. Take a topological surface $S$ of genus $g$ with $n+b$ boundary components. Select $n$ of these boundary components and label them from 1 to $n$. Place $k_i$ intervals along the $i$th boundary component and use the orientation on the boundary coming from the orientation on the surface to cyclically order these boundary intervals. Picking representatives of these cyclic orderings provides a way to label all the intervals from 1 to $K:=\sum_{i=1}^n k_i$. The tensor \eqref{eqn_tensor} is just the tensor that is associated to the surface $S\in M_{g,n,K}$ by the open topological field theory arising from the differential graded Frobenius algebra $A$.

This tensor is well-defined and does not depend on the various choices for labeling the intervals that we made above, as we will now explain. Firstly, suppose that we choose a different labeling of the boundary components. Using an orientation preserving diffeomorphism, we can permute any two boundary components of the surface. This means that any two surfaces in $M_{g,n,K}$ that arise from choosing different labels for the boundary components will be isomorphic, and hence give rise to the same tensor $\alpha_n^{g,b}$. Similarly, suppose that we choose a different representative for the cyclic ordering of the intervals on one boundary component. By using a diffeomorphism of the surface that rotates that boundary component, we can obtain any representative for that cyclic order of the intervals on that boundary component that we wish. In this way, we see that the definition of the tensor $\alpha_n^{g,b}$ does not depend on the choice of this representative.

Let us write down a concrete formula for this tensor. We may write the inverse inner product as $\innprod^{-1}=x_i\otimes y^i$, where the repeated index implies a sum. Now a formula for the tensor $\alpha_n^{0,0}$ is
\begin{multline} \label{eqn_tensorzero}
\alpha_n^{0,0}[(a_{11}\cdots a_{1k_1})\otimes\cdots\otimes(a_{n1}\cdots a_{nk_n})]= \\
(-1)^pt_n(x_{i_n},\ldots,x_{i_1})t_{k_1+\dots+k_n+n}(y^{i_1},a_{11},\ldots,a_{1k_1},\ldots,y^{i_n},a_{n1},\ldots,a_{nk_n})
\end{multline}
where the repeated indices again indicate a summation. Here, the tensor $t_n:A^{\otimes n}\to\gf$ is given by
\begin{equation} \label{eqn_product}
t_n(a_1,\ldots,a_n):=\langle a_1\cdots a_{n-1},a_n \rangle
\end{equation}
and the sign is given by the formula
\begin{equation} \label{eqn_tensorsign}
p:=\sum_{r=1}^{n-1} |y^{i_{r+1}}|(|a_{11}|+\cdots+|a_{1k_1}|+\cdots+|a_{r1}|+\cdots+|a_{rk_r}|).
\end{equation}
To describe the tensors $\alpha_n^{g,b}$, we define maps $\beta,\gamma:A\to A$ by
\begin{equation} \label{eqn_beta}
\beta(a):= x_iy^ia
\end{equation}
and
\begin{equation} \label{eqn_gamma}
\gamma(a):=(-1)^{x_jy^i}x_ix_jy^iy^ja.
\end{equation}
This allows us to give a concrete formula for $\alpha_n^{g,b}$:
\begin{multline} \label{eqn_tensorformula}
\alpha_n^{g,b}[(a_{11}\cdots a_{1k_1})\otimes\cdots\otimes(a_{n1}\cdots a_{nk_n})]=\\
\alpha_n^{0,0}[(\beta^b\gamma^g(a_{11})\cdots a_{1k_1})\otimes\cdots\otimes(a_{n1}\cdots a_{nk_n})]
\end{multline}
In fact, the formula remains true if we apply the map $\beta^b\gamma^g$ to any one of the arguments of $\alpha_n^{0,0}$.

These tensors satisfy certain identities which we will use in our construction of a cocycle on the moduli space.

\begin{lemma} \label{lem_otftid}
The tensors $\alpha_n^{g,b} \in T(A^*)^{\otimes n}$ satisfy the following identities:
\begin{enumerate}
\item \label{item_tensorsymmetric}
The tensors \eqref{eqn_tensor} are $\mathbb{S}_n$-invariant,
\begin{multline*}
\alpha_n^{g,b}[(a_{11}\cdots a_{1k_1})\otimes\cdots\otimes(a_{n1}\cdots a_{nk_n})] =\\
\pm \alpha_n^{g,b}[(a_{\sigma(1)1}\cdots a_{\sigma(1)k_{\sigma(1)}})\otimes\cdots\otimes(a_{\sigma(n)1}\cdots a_{\sigma(n)k_{\sigma(n)}})]
\end{multline*}
for any $\sigma\in\mathbb{S}_n$, where the sign $\pm$ is given by the usual Koszul sign rule.
\item \label{item_tensorcyclic}
The tensors \eqref{eqn_tensor} are cyclically invariant in each one of the $n$ arguments,
\begin{multline*}
\alpha_n^{g,b}[(a_{11}\cdots a_{1k_1})\otimes\cdots\otimes(a_{n1}\cdots a_{nk_n})] =\\
\pm\alpha_n^{g,b}[(a_{1\sigma_1(1)}\cdots a_{1\sigma_1(k_1)})\otimes\cdots\otimes(a_{n\sigma_n(1)}\cdots a_{n\sigma_n(k_n)})]
\end{multline*}
where $\sigma_i\in\mathbb{Z}/k_i\mathbb{Z}$ is any collection of cyclic permutations; again, the sign $\pm$ is given by the Koszul sign rule.
\item \label{item_tensorboundarystrip}
\begin{multline*}
\alpha_n^{g,b}[(x_ia_{11}\cdots a_{1k_1}y^ia_{21}\cdots a_{2k_2})\otimes\cdots\otimes(a_{(n+1)1}\cdots a_{(n+1)k_{n+1}})] =\\
(-1)^p\alpha_{n+1}^{g,b}[(a_{11}\cdots a_{1k_1})\otimes(a_{21}\cdots a_{2k_2})\otimes\cdots\otimes(a_{(n+1)1}\cdots a_{(n+1)k_{n+1}})]
\end{multline*}
where $p:=|y^i|(|a_{11}|+\cdots+|a_{1k_1}|)$.
\item \label{item_tensorboundarygenus}
\begin{multline*}
\alpha_n^{g,b}[(a_{11}\cdots a_{1k_1}x_i)\otimes(y^ia_{21}\cdots a_{2k_2})\otimes\cdots\otimes(a_{n1}\cdots a_{nk_n})] =\\
\alpha_{n-1}^{g+1,b}[(a_{11}\cdots a_{1k_1}a_{21}\cdots a_{2k_2})\otimes(a_{31}\cdots a_{3k_3})\otimes\cdots\otimes(a_{n1}\cdots a_{nk_n})]
\end{multline*}
\item \label{item_tensorglue}
\begin{multline*}
\alpha_{n_1}^{g_1,b_1}[(a_{11}\cdots a_{1k_1})\otimes\cdots\otimes(a_{n_11}\cdots a_{n_1k_{n_1}}x_i)]\alpha_{n_2}^{g_2,b_2}[(y^ib_{11}\cdots b_{1l_1})\otimes\cdots\otimes(b_{n_21}\cdots b_{n_2l_{n_2}})] =\\
\alpha_{n_1+n_2-1}^{g_1+g_2,b_1+b_2}[(a_{11}\cdots a_{1k_1})\otimes\cdots\otimes(a_{n_11}\cdots a_{n_1k_{n_1}}b_{11}\cdots b_{1l_1})\otimes\cdots\otimes(b_{n_21}\cdots b_{n_2l_{n_2}})]
\end{multline*}
\item \label{item_tensorclosed}
The tensors \eqref{eqn_tensor} are $d$-closed.
\end{enumerate}
\end{lemma}

\begin{proof}
Items \eqref{item_tensorsymmetric} and \eqref{item_tensorcyclic} concerning the invariance properties of the tensor $\alpha_n^{g,b}$ follow from exactly the same arguments that established that the tensor $\alpha_n^{g,b}$ was well-defined.

Items \eqref{item_tensorboundarystrip}, \eqref{item_tensorboundarygenus} and \eqref{item_tensorglue} follow from the axioms of an open topological field theory which state that the map yielded by the open topological field theory that is obtained from a surface after gluing that surface along two of its boundary intervals is the same as the map that is yielded by the original surface via the open topological field theory, after contracting those tensors corresponding to those intervals using the inverse inner product $\innprod=x_i\otimes y^i$. These identities are illustrated graphically by the following figures.

\begin{figure}[htp]
\centering
\includegraphics{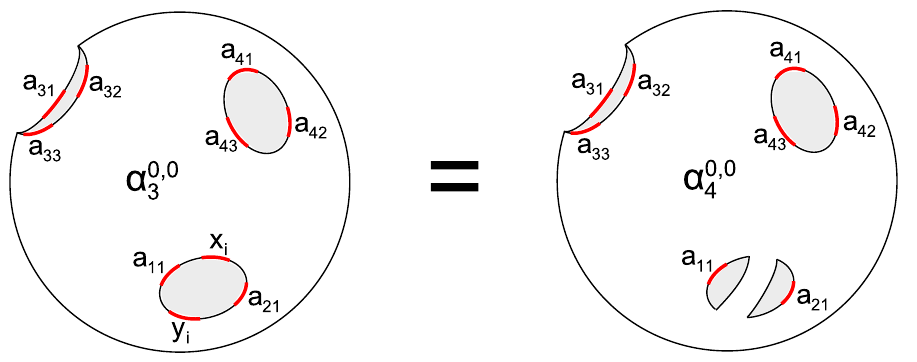}
\caption{The topological field theory axiom for gluing intervals located in the same boundary component described by Item \eqref{item_tensorboundarystrip}.}
\end{figure}

\clearpage

\begin{figure}[htp]
\centering
\includegraphics{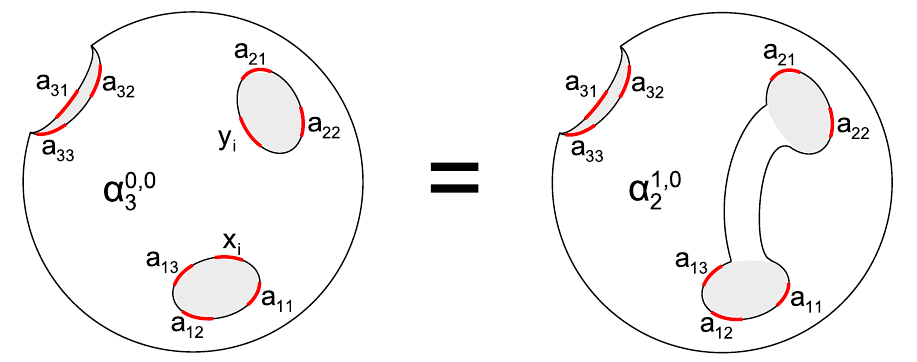}
\caption{The topological field theory axiom for gluing intervals located in different boundary components described by Item \eqref{item_tensorboundarygenus}.}
\end{figure}

\begin{figure}[htp]
\centering
\includegraphics{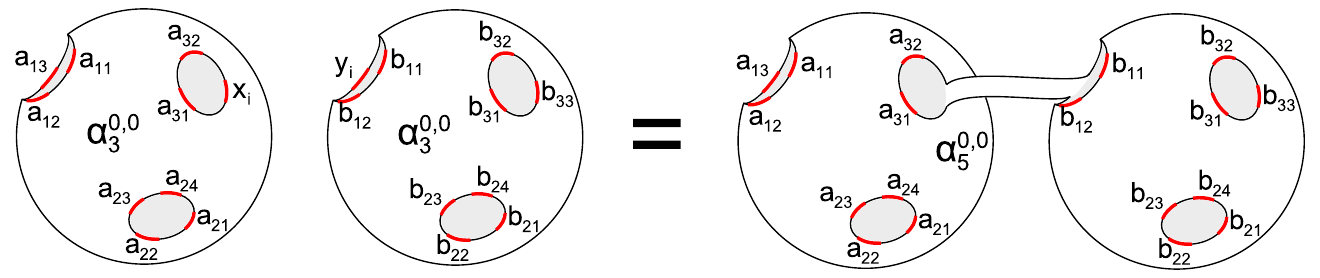}
\caption{The topological field theory axiom for gluing intervals located in different surfaces described by Item \eqref{item_tensorglue}.}
\end{figure}

Item \ref{item_tensorclosed} is a consequence of Definition \ref{def_modop} which states that morphisms of modular operads are supposed to commute with the differentials.
\end{proof}

In the preceding discussion, it was necessary to assume that the inner product on our Frobenius algebra $A$ was \emph{even}. A slightly different approach is required if this inner product is \emph{odd}. In this case, it is possible to formulate the notion of a \emph{twisted} open topological field theory, again using the language of modular operads; however, in this case this language can become quite cumbersome, so we will provide a direct description instead.

Given a differential graded Frobenius algebra $A$ with an \emph{odd} inner product, we define the tensors $\alpha_n^{g,b}\in T(A^*)^{\otimes n}$ by the same formulae \eqref{eqn_tensorzero}--\eqref{eqn_tensorformula} that described this tensor in the case of an \emph{even} inner product, the only difference being that now the sign \eqref{eqn_tensorsign} is given by the formula
\[ p:=\sum_{r=1}^{n-1} |y^{i_{r+1}}|(|a_{11}|+\cdots+|a_{1k_1}|+\cdots+|a_{r1}|+\cdots+|a_{rk_r}|+r). \]
These tensors satisfy identities completely analogous to those of Lemma \ref{lem_otftid}, the only slight difference being in Item \ref{item_tensorsymmetric} in that the tensor $\alpha_n^{g,b}$ is \emph{skew}-symmetric; that is to say that it is invariant under the \emph{signed} action of $\mathbb{S}_n$. Furthermore, one may note that in the case of an \emph{odd} inner product we have $\beta^2=0$, so that in this case the tensors $\alpha_n^{g,b}$ are zero for $b\geq 2$.

\section{Classical Batalin-Vilkovisky formalism} \label{sec_cbvform}

In this section we briefly recall classical theorems of the Batalin-Vilkovisky formalism and superintegral calculus, as a precursor to the formulation of a noncommutative analogue described in the next section. This material is completely standard, cf. for example \cite{bvgeom}; however, we shall refer to \cite{bvform} for the formulation of these standard results which are most convenient for our purposes.

\subsection{Superintegrals} \label{sec_superintegrals}

Suppose that we fix a nondegenerate quadratic superfunction $\sigma$ on $\mathbb{R}^{n|2m}$ of the form
\begin{equation} \label{eqn_quadratic}
\sigma=\frac{1}{2}\sum_{i=1}^k x_i^2 - \frac{1}{2}\sum_{i=k+1}^n x_i^2 + \sum_{i=1}^m \xi_{2i-1}\xi_{2i}.
\end{equation}
Given any polynomial superfunction $f(x,\xi)$ on $\mathbb{R}^{n|2m}$, there is a way to make sense of integrals of the form
\[ \int_{\mathbb{R}^{n|m}} f(x,\xi)e^{-\sigma(x,\xi)} d\mathbf{x}d\mathbf{\xi}. \]
Recall that the integration of such a superfunction with respect to the \emph{odd} variables $\xi_i$ is defined by stipulating that the integral operators $\int d\xi_i$ coincide with the differential operators $\frac{\partial}{\partial\xi_i}$. In this way, integrating out the odd variables amounts to picking out the coefficient of the $\xi_1\cdots\xi_{2m}$ term of the polynomial. We are then left with the problem of making sense of integrals of the form
\[ \int_{\mathbb{R}^n}f(x)e^{-\frac{1}{2}(\sum_{i=1}^k x_i^2 - \sum_{i=k+1}^n x_i^2)}, \]
where $f(x)$ is a polynomial in $x$. This is done in the standard way, by means of the Wick rotation; for instance $\int_{-\infty}^{\infty}e^{\frac{1}{2}x^2}dx=-i\sqrt{2\pi}$.

\subsection{Classical Batalin-Vilkovisky formalism}

Here we will recall the basic workings of the Batalin-Vilkovisky formalism. Let $(V,\innprod)$ be a symplectic vector space with an \emph{odd} symplectic form. From such a space there arises a canonical differential graded Lie algebra
\[ \mathfrak{p}[V]:=\gf[h^{-1},h]\otimes S(V^*) \]
which comes equipped with the canonical Poisson anti-bracket $\{-,-\}$ and super-Laplacian $h\Delta$ induced by the symplectic form $\innprod$.

There is a map of complexes
\begin{equation} \label{eqn_BVmult}
\mathrm{M}:C_\bullet(\mathfrak{p}[V]) \to \mathfrak{p}[V]
\end{equation}
defined by the equation $\mathrm{M}(p_1\otimes\cdots\otimes p_n):= h^{-n}p_1\cdots p_n$. This fact follows from the standard defining identity for the bracket $\{-,-\}$ and super-Laplacian $\Delta$,
\begin{equation} \label{eqn_BVidentity}
\Delta(xy)=\Delta(x)y + (-1)^x x\Delta(y)+\{x,y\}.
\end{equation}

Now let us fix a quadratic superfunction $\sigma$ on $V$ and consider the differential graded Lie algebra
\[ \mathfrak{p}_\sigma[V]\subset\gf[h^{-1},h]\otimes\widehat{S}(V^*) \]
consisting of functions of the form
\[ f(x,\xi)e^{-\sigma(x,\xi)}, \]
where $f(x,\xi)$ is a polynomial superfunction; the differential $h\Delta$ and bracket $\{-,-\}$ are defined as before. Note that in this notation we have $\mathfrak{p}[V]=\mathfrak{p}_0[V]$.

We say that a subspace $L\subset V$ is a \emph{Lagrangian subspace} if it is an isotropic subspace of maximal dimension. Two Lagrangian subspaces have the same \emph{type} if they have the same superdimension, that is to say that their odd and even parts both have the same dimension. Given any choice of a Lagrangian subspace $L\subset V$ on which the quadratic superfunction $\sigma$ is \emph{nondegenerate}, we can define an $\gf[h,h^{-1}]$-linear map
\begin{equation} \label{eqn_bvcocycle}
\begin{array}{ccc}
\mathfrak{p}_\sigma[V] & \to & \mathbb{R}[h^{-1},h] \\
f(x,\xi)e^{-\sigma(x,\xi)} & \mapsto & \frac{\int_L f(x,\xi)e^{-\sigma(x,\xi)} d\mathbf{x}d\mathbf{\xi}}{\int_L e^{-\sigma(x,\xi)}d\mathbf{x}d\mathbf{\xi}}
\end{array}
\end{equation}
In order to make sense of the integrals appearing on the right, we identify our Lagrangian subspace $L$ with a space $\mathbb{R}^{n|2m}$ on which the quadratic superfunction $\sigma$ takes the canonical form \eqref{eqn_quadratic}. These integrals may then be evaluated using the method described in Section \ref{sec_superintegrals}. Under this definition it may seem that the quantity on the right depends on how we identify $L$ with $\mathbb{R}^{n|2m}$; however, there is a combinatorial formula, known as Wick's formula, which states that the quantity on the right may be evaluated by a sum whose terms are defined by contracting the tensors of the polynomial $f(x,\xi)$ with the inverse inner product corresponding to $\sigma$. It is this formula that implies that the quantity on the right is well-defined.

\begin{theorem} \label{thm_BVcocycle}
The above map \eqref{eqn_bvcocycle} is a $h\Delta$-cocycle. Furthermore, suppose that $\sigma$ satisfies the quantum master equation $\Delta(\sigma)=\frac{1}{2}\{\sigma,\sigma\}$; then if $L$ and $L'$ are two Lagrangian subspaces of the same type on which $\sigma$ is nondegenerate, the corresponding cocycles are cohomologous.
\end{theorem}

\begin{proof}
This theorem is a standard result of the Batalin-Vilkovisky formalism. For convenience we shall refer to the previous paper \cite{bvform}. That this map is a $h\Delta$-cocycle is Corollary 5.17 of \cite{bvform}. The fact that two Lagrangian subspaces of the same type give rise to cohomologous cocycles can be shown in precisely the same way that part (2) of Theorem 6.17 of \cite{bvform} is proven.
\end{proof}

\section{Noncommutative Batalin-Vilkovisky formalism and the main construction} \label{sec_ncbv}

The central object of the classical Batalin-Vilkovisky formalism is the differential graded Lie algebra $\mathfrak{p}[V]$ which encodes its underlying commutative geometry. In what follows we replace this object by the noncommutative object $\Lambda_{\gamma,\nu}[V]$ which plays the same role. The relationship between the two is expressed by a simple map \eqref{eqn_nctocomm} from $\Lambda_{\gamma,\nu}[V]$ to $\mathfrak{p}[V]$ which allows us to transfer all the classical theorems from the Batalin-Vilkovisky formalism to our new setting and consider the space $\Lambda_{\gamma,\nu}[V]$ as a space of noncommutative superfunctions that may be integrated in exactly the same way as described previously. This straightforward result is established later in the section in Proposition \ref{prop_nctocomm}.

Working in this noncommutative context we develop a construction which produces from a contractible differential graded Frobenius algebra, a Chevalley-Eilenberg cohomology class of the differential graded Lie algebra $\Lambda_{\gamma,\nu}[V]$ and hence, by Theorem \ref{thm_modspc}, a cohomology class on the corresponding compactification of the moduli space. This construction may be regarded as a formulation of the construction described by Chuang-Lazarev in \cite{dualfeyn} within the framework of the Batalin-Vilkovisky formalism, the latter being essentially a development of a construction described by Kontsevich in \cite{kontfeyn}. The main advantage of our approach is that it will enable us to perform concrete calculations, which we describe more fully later in the paper. The key ingredient of our approach which ensures that the \emph{cochain} that we construct from our Frobenius algebra is actually a \emph{cocycle} is the list of identities enumerated in Lemma \ref{lem_otftid} which correspond to the axioms of an open topological field theory; in a sense, the axioms of an open topological field theory are equivalent to the statement that the expressions that we write down give rise to a cocycle and not just a cochain.

\subsection{Map of differential graded Lie algebras}

In this section we describe how any differential graded Frobenius algebra gives rise to a map between two differential graded Lie algebras playing the role of noncommutative superfunctions in the Batalin-Vilkovisky formalism. The existence and properties of this map depend upon the axioms of an open topological field theory.

Let $(V,\innprod_V)$ be a symplectic vector space with an \emph{odd} symplectic form, and let $A$ be a differential graded Frobenius algebra whose nondegenerate symmetric bilinear form $\innprod_A$ is \emph{even}. On the tensor product $V\otimes A$ we can define an odd symplectic form by the formula
\[ \langle v_1\otimes a_1,v_2\otimes a_2 \rangle_{V\otimes A} = (-1)^{a_1v_2}\langle v_1,v_2\rangle_V \langle a_1,a_2 \rangle_A. \]

To any such pair of symplectic vector space $V$ and differential graded Frobenius algebra $A$, we can associate a map of differential graded Lie algebras
\begin{equation} \label{eqn_dglamap}
\Phi_A:\Lambda_{\gamma,\nu}[V]\to\Lambda_{\gamma,\nu}[V\otimes A]
\end{equation}
using the tensors $\alpha_n^{g,b} \in T(A^*)^{\otimes n}$ of \eqref{eqn_tensor}. This map $\Phi_A$ is defined uniquely by the commutative diagram
\begin{equation} \label{fig_dglamap}
\xymatrix{V^{*\otimes k_1}\otimes\cdots\otimes V^{*\otimes k_n} \ar[rrr]^-{-\otimes\alpha_n^{g,b}|_{[k_1,\ldots,k_n]}} \ar[dd]_{\pi^{g,b}} &&& (V^{*\otimes k_1}\otimes\cdots\otimes V^{*\otimes k_n})\otimes(A^{*\otimes k_1}\otimes\cdots\otimes A^{*\otimes k_n}) \ar@{=}[d] \\ &&& (V\otimes A)^{*\otimes k_1}\otimes\cdots\otimes (V\otimes A)^{*\otimes k_n} \ar[d]^{\pi^{g,b}} \\ \Lambda_{\gamma,\nu}[V] \ar[rrr]^{\Phi_A} &&& \Lambda_{\gamma,\nu}[V\otimes A] }
\end{equation}
where the map
\[\pi^{g,b}:V^{*\otimes k_1}\otimes\cdots\otimes V^{*\otimes k_n}\to\Lambda_{\gamma,\nu}[V]\subset\gf[\gamma,\nu]\otimes S(\mathfrak{h}_{\geq 1}[V])\]
is defined by the formula
\[ \pi^{g,b}[(x_{11}\otimes\cdots\otimes x_{1k_1})\otimes\cdots\otimes (x_{n1}\otimes\cdots\otimes x_{nk_n})]:= \gamma^g\nu^b[(x_{11}\cdots x_{1k_1})\cdots (x_{n1}\cdots x_{nk_n})],\]
for any $g,b\geq 0$. The fact that such a map $\Phi_A$ \emph{exists} which makes the above diagram commutative for all $g,b\geq 0$ is a consequence of items \eqref{item_tensorsymmetric} and \eqref{item_tensorcyclic} of Lemma \ref{lem_otftid} describing the symmetries of the tensor $\alpha_n^{g,b}$.

\begin{theorem} \label{thm_dglamap}
The map $\Phi_A$ is a map of differential graded Lie algebras.
\end{theorem}

\begin{proof}
This is a consequence of the axioms of an open topological field theory. In fact the statement that $\Phi_A$ is a map of differential graded Lie algebras is equivalent to identities \eqref{item_tensorboundarystrip}--\eqref{item_tensorglue} of Lemma \ref{lem_otftid}, as we will now explain.

The differentials $d$ on $\Lambda_{\gamma,\nu}[V]$ and $\Lambda_{\gamma,\nu}[V\otimes A]$ may be separated into two parts
\[ d:=\gamma\cdot\delta + \Delta \]
where $\delta$ is induced by the bracket \eqref{eqn_bracket} on $\mathfrak{h}[V]$ and $\Delta$ is induced by the cobracket \eqref{eqn_cobracket} on $\mathfrak{h}[V]$. The fact that the map $\Phi_A$ commutes with the differential $\Delta$ is a consequence of identity \eqref{item_tensorboundarystrip} of Lemma \ref{lem_otftid}. The fact that the map $\Phi_A$ commutes with the differential $\gamma\cdot\delta$ is a consequence of identity \eqref{item_tensorboundarygenus} of the same lemma. Lastly, the fact that $\Phi_A$ respects the brackets \eqref{eqn_CEbracket} on $\Lambda_{\gamma,\nu}[V]$ and $\Lambda_{\gamma,\nu}[V\otimes A]$ induced by the brackets \eqref{eqn_bracket} on $\mathfrak{h}[V]$ and $\mathfrak{h}[V\otimes A]$ is a consequence of identity \eqref{item_tensorglue} of the same Lemma.

These facts may seem surprising initially, so let us demonstrate the first one in a simple example. We may write the tensor $\alpha_1^{0,0}|_{[k]}\in A^{*\otimes k}$ as
\[ \alpha_1^{0,0}|_{[k]} = \sum_r a^r_1\otimes\cdots\otimes a^r_k. \]
Now consider the tensor $\mathbf{x}:=x_1\cdots x_k\in\mathfrak{h}[V]\subset\Lambda_{\gamma,\nu}[V]$, then

\begin{displaymath}
\begin{split}
\Phi(\Delta(\mathbf{x})) &= \sum_{i<j}\pm\langle x_i,x_j \rangle^{-1}\pi^{0,0}\left[\begin{subarray}{c} (x_{i+1}\otimes\cdots\otimes x_{j-1}) \otimes (x_{j+1}\otimes\cdots\otimes x_k \otimes x_1\otimes\cdots\otimes x_{i-1}) \\ \otimes \\ \alpha_2^{0,0}|_{[j-i-1,k-j+i-1]} \end{subarray}\right] \\
\Delta(\Phi(\mathbf{x})) &= \Delta\pi^{0,0}\left[\begin{subarray}{c} (x_1\otimes\cdots\otimes x_k) \\ \otimes \\ \alpha_1^{0,0}|_{[k]} \end{subarray}\right] \\
&= \sum_r\sum_{i<j}\pm\langle x_i,x_j \rangle^{-1}\langle a^r_i,a^r_j \rangle^{-1}\pi^{0,0}\left[ \begin{subarray}{c}(x_{i+1}\otimes\cdots\otimes x_{j-1}) \otimes (x_{j+1}\otimes\cdots\otimes x_k \otimes x_1\otimes\cdots\otimes x_{i-1}) \\ \otimes \\ (a^r_{i+1}\otimes\cdots\otimes a^r_{j-1}) \otimes (a^r_{j+1}\otimes\cdots\otimes a^r_k\otimes a^r_1\otimes\cdots\otimes a^r_{i-1}) \end{subarray}\right]
\end{split}
\end{displaymath}
It follows from identity \eqref{item_tensorboundarystrip} of Lemma \ref{lem_otftid} and the fact that the tensor $\alpha_1^{0,0}|_{[k]}$ is cyclically symmetric that these two expressions coincide.
\end{proof}

Suppose instead that we start with a symplectic vector space $(V,\innprod_V)$ whose symplectic form is \emph{even} and a differential graded Frobenius algebra $A$ whose bilinear form $\innprod_A$ is \emph{odd}. In such a case we may also construct a map of differential graded Lie algebras
\[\Phi_A:\Lambda_{\gamma,\nu}[V]\to\Lambda_{\gamma,\nu}[V\otimes A]\]
in exactly the same way as we did above by instead using the tensors $\alpha_n^{g,b} \in T(A^*)^{\otimes n}$ defined in this situation at the end of Section \ref{sec_otft}. In either case, the map $\Phi_A$ induces a map
\[ \Phi_A:C_\bullet(\Lambda_{\gamma,\nu}[V]) \to C_\bullet(\Lambda_{\gamma,\nu}[V\otimes A]) \]
on the Chevalley-Eilenberg complexes.

\subsection{Hodge decomposition}
Here we recall the notion of an abstract Hodge decomposition from \cite{minmod}.

\begin{defi}
Suppose we have a differential graded Frobenius algebra $A$. An \emph{abstract Hodge decomposition} of $A$ consists of a pair of operators $s,\pi:A\to A$ such that
\begin{equation}
\begin{split}
ds+sd &= \id - \pi \\
s^2 &= 0 \\
\pi^2 &= \pi \\
d\pi &= \pi d =0 \\
\pi s &= s\pi=0 \\
\langle s(a),b \rangle &= (-1)^a \langle a, s(b) \rangle \\
\langle \pi(a),b \rangle &= \langle a,\pi(b) \rangle
\end{split}
\end{equation}
\end{defi}

Such an abstract Hodge decomposition is equivalent to a decomposition of $A$
\[ A=\im(d) \oplus \im(\pi) \oplus \im(s) \]
into an acyclic subspace $\im(d)\oplus \im(s)$ and a subspace $\im(\pi)\cong H(A)$. Any differential graded Frobenius algebra $A$ can be decomposed in this manner, in a noncanonical way \cite{feynmod}.

\subsection{Chevalley-Eilenberg cocycles} \label{sec_cocycles}

Let $(V,\innprod_V)$ be a symplectic vector space whose symplectic form is \emph{odd}. We are now in a position to construct a cocycle on the Chevalley-Eilenberg complex of $\Lambda_{\gamma,\nu}[V]$, and concomitantly on the Looijenja compactification of the moduli space. Furthermore, we will describe under what conditions our construction gives rise to a Chevalley-Eilenberg cocycle of the differential graded Lie algebra $\Lambda_{\gamma}[V]$ whose homology computes the homology of the Kontsevich compactification of the moduli space. Before we do this however, we formulate the following simple result describing the relationship between the noncommutative geometry used in this section and the commutative geometry used in the previous section on the classical Batalin-Vilkovisky formalism.

\begin{prop} \label{prop_nctocomm}
There is a map of differential graded Lie algebras
\begin{equation} \label{eqn_nctocomm}
\mathrm{N}:\Lambda_{\gamma,\nu}[V] \to \mathfrak{p}[V]
\end{equation}
defined by the equation
\[ \mathrm{N}(\gamma^i\nu^j [(x_{11}\cdots x_{1k_1})\cdots(x_{n1}\cdots x_{nk_n})]) := h^{2i+j+n-1}x_{11}\cdots x_{1k_1}\cdots x_{n1}\cdots x_{nk_n}. \]
\end{prop}

\begin{proof}
It is a straightforward consequence of the standard identity \eqref{eqn_BVidentity} for the differential graded Lie algebra $\mathfrak{p}[V]$ that $N$ is a map of differential graded Lie algebras.
\end{proof}

Now we proceed with our construction. Suppose that we now take a differential graded Frobenius algebra $A$ whose nondegenerate symmetric bilinear form $\innprod_A$ is \emph{even}. From this we may construct an even symmetric \emph{degenerate} bilinear form $\innprod_d$ on $V\otimes A$:
\begin{equation} \label{eqn_degform}
\langle v_1\otimes a_1, v_2\otimes a_2 \rangle_d:=(-1)^{a_1 (v_2+1)}\langle v_1,v_2 \rangle_V \langle a_1,d(a_2) \rangle_A.
\end{equation}
We will denote the corresponding quadratic superfunction on $V\otimes A$ by $\sigma_d$.

\begin{lemma} \label{lem_brackvanish}
For any $\mathbf{x}:=\gamma^i\nu^j [(x_{11}\cdots x_{1k_1})\cdots(x_{n1}\cdots x_{nk_n})]\in\Lambda_{\gamma,\nu}[V]$,
\[ \{\sigma_d,\Phi_A(\mathbf{x})\}=0. \]
\end{lemma}

\begin{proof}
One may check by means of a straightforward calculation that
\begin{displaymath}
\begin{split}
\{\sigma_d,\Phi_A(\mathbf{x})\} &= \frac{1}{2}\left\{\pi^{0,0}\left[\begin{subarray}{c} \innprod_V \\ \otimes \\ \langle-,d-\rangle_A \end{subarray}\right],\pi^{i,j}\left[\begin{subarray}{c} (x_{11}\otimes\cdots\otimes x_{1k_1})\otimes\cdots\otimes(x_{n1}\otimes\cdots\otimes x_{nk_n}) \\ \otimes \\ \alpha_n^{i,j}|_{[k_1,\ldots,k_n]} \end{subarray}\right]\right\} \\
&= \pi^{i,j}\left[\begin{subarray}{c} (x_{11}\otimes\cdots\otimes x_{1k_1})\otimes\cdots\otimes(x_{n1}\otimes\cdots\otimes x_{nk_n}) \\ \otimes \\ d^*(\alpha_n^{i,j}|_{[k_1,\ldots,k_n]}) \end{subarray}\right] =0.
\end{split}
\end{displaymath}
Where the last line follows from Item \eqref{item_tensorclosed} of Lemma \ref{lem_otftid}.
\end{proof}

We may now proceed to give a definition of our cocycle in $C^\bullet(\Lambda_{\gamma,\nu}[V])$. Suppose now that our differential graded Frobenius algebra $A$ is \emph{contractible}, and that it comes equipped with a choice of a Hodge decomposition $s:A\to A$. Such a choice of Hodge decomposition is equivalent to the choice of a subspace $L=\im(s)\subset A$ that is isotropic with respect to $\innprod_A$ and which is complementary to the isotropic subspace $\im(d)$. Consequently, the bilinear form $\langle-,d-\rangle_A$ is \emph{nondegenerate} on $L$. This subspace $L$ gives rise to a Lagrangian subspace $\tilde{L}:=V\otimes L$ of the symplectic vector space $(V\otimes A,\innprod_{V\otimes A})$.

\begin{defi}
From this data we define a cochain $Q_A\in C^\bullet(\Lambda_{\gamma,\nu}[V])$ with coefficients in $\mathbb{R}[h]$ by the formula
\begin{equation} \label{eqn_cocycleformula}
Q_A(\mathbf{x}_1\cdots\mathbf{x}_k):= \frac{\int_{\tilde{L}} \mathrm{M}\mathrm{N}\Phi_A(\mathbf{x}_1\cdots\mathbf{x}_k)\cdot e^{-\frac{1}{h}\sigma_d}}{\int_{\tilde{L}} e^{-\frac{1}{h}\sigma_d}}.
\end{equation}
In other words, $Q_A$ is the composition
\[ Q_A:C_\bullet(\Lambda_{\gamma,\nu}[V]) \overset{\Phi_A}{\longrightarrow} C_\bullet(\Lambda_{\gamma,\nu}[V\otimes A]) \overset{\mathrm{N}}{\longrightarrow} C_\bullet(\mathfrak{p}[V\otimes A]) \overset{\mathrm{M}}{\longrightarrow} \mathfrak{p}[V\otimes A] \overset{f\mapsto fe^{-\frac{1}{h}\sigma_d}}{\longrightarrow} \mathfrak{p}_{\frac{1}{h}\sigma_d}[V\otimes A] \longrightarrow \mathbb{R}[h], \]
Where $\Phi_A$ is defined by \eqref{eqn_dglamap}, $\mathrm{N}$ is defined by \eqref{eqn_nctocomm}, $\mathrm{M}$ is defined by \eqref{eqn_BVmult} and the rightmost map is defined by \eqref{eqn_bvcocycle}.
\end{defi}

\begin{theorem} \label{thm_cocycle}
This cochain $Q_A$ is a \emph{cocycle} whose cohomology class does not depend on the choice of a Hodge decomposition for the algebra $A$.
\end{theorem}

\begin{proof}
From the identity \eqref{eqn_BVidentity} we deduce
\begin{equation} \label{eqn_vanish}
\begin{split}
Q_A\delta(\mathbf{x}_1\cdots\mathbf{x}_k) =& \frac{\int_{\tilde{L}} h\Delta[\mathrm{M}\mathrm{N}\Phi_A(\mathbf{x}_1\cdots\mathbf{x}_k)]\cdot e^{-\frac{1}{h}\sigma_d}}{\int_{\tilde{L}} e^{-\frac{1}{h}\sigma_d}} \\
=& \frac{\int_{\tilde{L}} h\Delta[\mathrm{M}\mathrm{N}\Phi_A(\mathbf{x}_1\cdots\mathbf{x}_k)\cdot e^{-\frac{1}{h}\sigma_d}]}{\int_{\tilde{L}} e^{-\frac{1}{h}\sigma_d}} \\
& \pm \frac{\int_{\tilde{L}} \mathrm{M}\mathrm{N}\Phi_A(\mathbf{x}_1\cdots\mathbf{x}_k)\cdot h\Delta[e^{-\frac{1}{h}\sigma_d}]}{\int_{\tilde{L}} e^{-\frac{1}{h}\sigma_d}} \\
& \pm \frac{1}{h}\frac{\int_{\tilde{L}} \{\mathrm{M}\mathrm{N}\Phi_A(\mathbf{x}_1\cdots\mathbf{x}_k),\sigma_d\}\cdot e^{-\frac{1}{h}\sigma_d}}{\int_{\tilde{L}} e^{-\frac{1}{h}\sigma_d}}
\end{split}
\end{equation}

Since $\sigma_d$ is even, it follows purely from reasons of degree that $\Delta(\sigma_d)=0$, and one can check that the equation $\{\sigma_d,\sigma_d\}=0$ is a consequence of the identity $d^2=0$. From these two facts we may deduce
\[ h\Delta(e^{-\frac{1}{h}\sigma_d})= \frac{1}{h}(h\Delta(\sigma_d)-\{\sigma_d,\sigma_d\})e^{-\frac{1}{h}\sigma_d} = 0. \]
It follows that the second term in \eqref{eqn_vanish} vanishes. The third term in \eqref{eqn_vanish} vanishes by Lemma \ref{lem_brackvanish} and the first term vanishes by Theorem \ref{thm_BVcocycle}. Consequently, $Q_A$ is a cocycle. Since a different choice of Hodge decomposition of $A$ simply amounts to an alternative choice of Lagrangian subspace $\tilde{L}\subset V\otimes A$, it follows again from Theorem \ref{thm_BVcocycle} that the cohomology class of this cocycle does not depend on the choice of a Hodge decomposition.
\end{proof}

\begin{rem}
Consequently, any differential graded Frobenius algebra $A$ whose bilinear form $\innprod_A$ is even leads to a well defined cohomology class in $C^\bullet(\Lambda_{\gamma,\nu}[V])$.
\end{rem}

There is a formula, known as ``Wick's formula'', which allows us to compute the integral \eqref{eqn_cocycleformula} by a sum whose terms are defined by applying the bilinear form $\innprod_d^{-1}$ to the polynomial $\mathrm{M}\mathrm{N}\Phi_A(\mathbf{x}_1\cdots\mathbf{x}_k)$ in all possible ways. One consequence of this formula is that the cocycle $Q_A$ is well-defined modulo the action of $\mathfrak{pe}[V]$, and gives rise to a \emph{relative} cocycle, which we will also denote by $Q_A$, in $C^\bullet(\Lambda_{\gamma,\nu}[V],\mathfrak{pe}[V])$.

Consider now the differential graded Lie algebra $\Lambda_{\gamma,\nu}^1:= \dilim{n}{\Lambda_{\gamma,\nu}[W^1_{n|n}]}$. One can check that the cocycles $Q_{A;n}$ in $C^\bullet(\Lambda_{\gamma,\nu}[W^1_{n|n}],\mathfrak{pe}[W^1_{n|n}])$ defined by \eqref{eqn_cocycleformula} are compatible with the natural maps between the $W^1_{n|n}$, and consequently give rise to a cocycle $Q_A\in C^\bullet(\Lambda_{\gamma,\nu}^1,\mathfrak{pe})$. By Theorem \ref{thm_modspc} this gives rise to a whole family of cocycles (with coefficients in $\gf[h]$) on the Looijenga compactifications $\lcmp$ of the moduli space for varying $g$ and $n$.

A similar remark applies in the situation that $A$ comes equipped with an \emph{odd} bilinear form. In this case, in exactly the same way as before and using exactly the same formulae, we may define cocycles $Q_{A;n,m}$ in $C^\bullet(\Lambda_{\gamma,\nu}[W^0_{2n|m}],\mathfrak{osp}[W^0_{2n|m}])$ which ultimately give rise to a cocycle $Q_A\in C^\bullet(\Lambda_{\gamma,\nu}^0,\mathfrak{osp})$. Again, by Theorem \ref{thm_modspc} this leads to a family of cohomology classes with twisted coefficients on the Looijenga compactifications of the moduli space.

Now we consider under what conditions this construction will produce a cocycle on the Kontsevich compactification $\kcmp$ of the moduli space. Consider the cochain $Q_A\in C^\bullet(\Lambda_{\gamma,\nu}[V])$. There is a natural inclusion
\[ \Lambda_{\gamma}[V]\hookrightarrow\Lambda_{\gamma,\nu}[V] \]
of vector spaces, which is \emph{not} a map of differential graded Lie algebras. Consequently, we may restrict this cochain to a cochain $Q_A\in C^\bullet(\Lambda_{\gamma}[V])$.

\begin{prop} \label{prop_betavanish}
If the map $\beta$ of Equation \eqref{eqn_beta} vanishes, then the cochain $Q_A\in C^\bullet(\Lambda_{\gamma}[V])$ is a \emph{cocycle} whose cohomology class does not depend on the choice of Hodge decomposition for $A$.
\end{prop}

\begin{rem}
In this case, by Theorem \ref{thm_modspc} this leads to a family of cohomology classes on the Kontsevich compactification $\kcmp$ of the moduli space.
\end{rem}

\begin{proof}
Let us denote the differential in $C_\bullet(\Lambda_{\gamma,\nu}[V])$ by $\delta_{\gamma,\nu}$ and the differential in $C_\bullet(\Lambda_{\gamma}[V])$ by $\delta_{\gamma}$, then the key formula is
\begin{equation} \label{eqn_betavanish}
\Phi_A[\delta_{\gamma}(\mathbf{x}_1\cdots\mathbf{x}_k)] = \Phi_A[\delta_{\gamma,\nu}(\mathbf{x}_1\cdots\mathbf{x}_k) -O(\nu\geq 1)]
\end{equation}
where $O(\nu\geq 1)$ is a sum of terms involving powers of $\nu$ of order $\geq 1$. However, since $\beta=0$, by Equation \eqref{eqn_tensorformula} we have
\[ \alpha_n^{g,b}=0,\quad \text{for all } b\geq 1. \]
By Figure \eqref{fig_dglamap} it follows from this formula that $\Phi_A[O(\nu\geq 1)]=0$. Now Equation \eqref{eqn_betavanish} and the arguments of Theorem \ref{thm_cocycle} suffice to establish the claim of the proposition.
\end{proof}

\section{Characteristic classes} \label{sec_char}

In this section we will recall elementary material about characteristic classes, as described in \cite{hamqme}, and apply this to the construction of \emph{homology} classes in compactifications of the moduli space of curves. The purpose of the sections that follow will be to describe what happens when we evaluate the cohomology classes produced in Section \ref{sec_ncbv} upon the homology classes produced in this section and to express the result as the perturbative expansion of some functional integral over a finite space of fields. The first notion we will need to recall is that of the Maurer-Cartan moduli space.

\begin{defi}
Let $\mathfrak{g}$ be a \emph{pronilpotent} differential graded Lie algebra. We define its Maurer-Cartan set by
\[ \mcs{\mathfrak{g}}:=\left\{x\in\mathfrak{g}_0:dx+\frac{1}{2}[x,x]=0\right\}.\]
To define the Maurer-Cartan moduli space $\mcm{\mathfrak{g}}$, we must consider the action of $\mathfrak{g}_1$ on $\mcs{\mathfrak{g}}$ defined by the equation
\[ \exp(y)\cdot x:=x+\sum_{n=0}^\infty\frac{1}{(n+1)!}[\ad y]^n(dy+[y,x]). \]
This action is well-defined as $\mathfrak{g}$ is pronilpotent. The Maurer-Cartan moduli space $\mcm{\mathfrak{g}}$ is the quotient of the Maurer-Cartan set $\mcs{\mathfrak{g}}$ by this action.
\end{defi}

In order to consider characteristic classes of Maurer-Cartan elements, we will need to introduce a completed version of the Chevalley-Eilenberg complex defined by Definition \ref{def_cecomp}.

\begin{defi}
The \emph{completed} Chevalley-Eilenberg complex $\widehat{C}_\bullet(\mathfrak{g})$ of a pronilpotent differential graded algebra $\mathfrak{g}$ is the complex whose underlying vector space is
\[ \widehat{C}_\bullet(\mathfrak{g}):=\prod_{i=0}^\infty [\mathfrak{g}^{\cotimes i}]_{\mathbb{S}_i} \]
and whose differential $\delta$ is defined by the same equation, Equation \eqref{eqn_cediff}, as in Definition \ref{def_cecomp}. The homology of this complex will be referred to as the \emph{completed Chevalley-Eilenberg homology of $\mathfrak{g}$ with trivial coefficients} and denoted by $\widehat{H}_\bullet(\mathfrak{g})$. The notion of the completed Chevalley-Eilenberg complex extends in an obvious manner to \emph{relative} Chevalley-Eilenberg homology.
\end{defi}

There is a well-known way \cite{ss} to produce homology classes in the completed Chevalley-Eilenberg complex of a pronilpotent differential graded Lie algebra, by exponentiating elements in the Maurer-Cartan set. Given a Maurer-Cartan element $x\in\mcs{\mathfrak{g}}$, we define its characteristic class $\tilde{\ch}(x)\in\widehat{C}_\bullet(\mathfrak{g})$ by the equation
\[ \tilde{\ch}(x):=\exp(x)=1 + x + \frac{1}{2}x\cdot x +\cdots +\frac{1}{n!}x^n+\cdots. \]
One can verify \cite{hamqme} that $\tilde{\ch}(x)$ is a cycle, and that equivalent Maurer-Cartan elements produces homologous cycles. This leads to the following theorem.

\begin{theorem} \label{thm_char}
For any pronilpotent differential graded algebra $\mathfrak{g}$, the above construction induces a well-defined map
\[ \tilde{\ch}:\mcm{\mathfrak{g}} \to \widehat{H}_\bullet(\mathfrak{g}). \]
\end{theorem}
\noproof

In order to apply the above ideas to the construction of classes in compactifications of the moduli space, we need to first introduce pronilpotent versions of the differential graded Lie algebras described in Section \ref{sec_ncgeom}. For a given symplectic vector space $(V,\innprod)$, the differential graded Lie algebras $\Lambda_{\gamma,\nu}[V]$ and $\Lambda_\gamma[V]$ have a natural filtration in which the deformation parameter $\gamma$ has order 2, the deformation parameter $\nu$ has order 1, and a cyclic word in $\mathfrak{h}[V]$ has order $i$.

\begin{defi}
The pronilpotent differential graded Lie algebras $\widehat{\Lambda}_{\gamma,\nu}[V]$ and $\widehat{\Lambda}_\gamma[V]$ are defined to be the natural completions of $\Lambda_{\gamma,\nu}[V]$ and $\Lambda_{\gamma}[V]$ with respect to the above filtrations.
\end{defi}

The filtration on $\Lambda_{\gamma,\nu}[V]$ naturally induces filtrations on the direct limits
\begin{displaymath}
\begin{split}
\Lambda^1_{\gamma,\nu} & :=\dilim{n}{\Lambda_{\gamma,\nu}[W^1_{n|n}]} \\
\Lambda^0_{\gamma,\nu} & :=\dilim{n,m}{\Lambda_{\gamma,\nu}[W^0_{2n|m}]} \\
\end{split}
\end{displaymath}
which allow us to consider their completions $\widehat{\Lambda}^1_{\gamma,\nu}$ and $\widehat{\Lambda}^0_{\gamma,\nu}$ with respect to these filtrations. The same remarks apply to $\Lambda_{\gamma}[V]$ and its direct limits. The following theorem is a simple consequence of Theorem \ref{thm_modspc}.

\begin{theorem} \label{thm_cmdspc}
If the prefix $\mathcal{P}$ denotes primitive homology as before, there are isomorphisms in homology:
\begin{displaymath}
\begin{split}
\mathcal{P}\widehat{H}_\bullet(\widehat{\Lambda}_{\gamma,\nu}^1,\mathfrak{pe}) & \cong \prod_{\chi=-1}^{-\infty} \left[\bigoplus_{\begin{subarray}{c} g\geq 0, \ n\geq 1: \\ \chi=2-2g-n  \end{subarray}} H_\bullet(\lcmp/\mathbb{S}_n) \right] \\
\mathcal{P}\widehat{H}_\bullet(\widehat{\Lambda}_{\gamma}^1,\mathfrak{pe}) & \cong \prod_{\chi=-1}^{-\infty} \left[\bigoplus_{\begin{subarray}{c} g\geq 0, \ n\geq 1: \\ \chi=2-2g-n  \end{subarray}} H_\bullet([\kcmp\times\Delta_n^\circ]^\infty/\mathbb{S}_n) \right]
\end{split}
\end{displaymath}
\end{theorem}
\noproof

\begin{rem}
As before, a similar theorem applies to the differential graded Lie algebras $\widehat{\Lambda}^0_{\gamma,\nu}$ and $\widehat{\Lambda}^0_{\gamma}$ if we consider the homology of the moduli space \emph{with twisted coefficients}.
\end{rem}

This leads to the following natural corollary of Theorems \ref{thm_char} and \ref{thm_cmdspc}.

\begin{cor} \label{cor_char}
Given a symplectic vector space $(V,\innprod)$ with an odd symplectic form, the characteristic class construction yields natural maps
\begin{displaymath}
\begin{split}
\ch: & \mcm{\widehat{\Lambda}_{\gamma,\nu}[V]} \to \prod_{\chi=-1}^{-\infty} \left[\bigoplus_{\begin{subarray}{c} g\geq 0, \ n\geq 1: \\ \chi=2-2g-n  \end{subarray}} H_\bullet(\lcmp/\mathbb{S}_n) \right] \\
\ch: & \mcm{\widehat{\Lambda}_{\gamma}[V]} \to \prod_{\chi=-1}^{-\infty} \left[\bigoplus_{\begin{subarray}{c} g\geq 0, \ n\geq 1: \\ \chi=2-2g-n  \end{subarray}} H_\bullet([\kcmp\times\Delta_n^\circ]^\infty/\mathbb{S}_n) \right]
\end{split}
\end{displaymath}
producing homology classes in compactifications of the moduli space of curves. Here $\ch$ denotes the projection of $\tilde{\ch}$ onto primitive homology. If, alternatively, the symplectic form is even, this construction produces classes in homology with twisted coefficients.
\end{cor}

\section{Pairing of classes and perturbative expansions} \label{sec_pairing}

In this paper we have considered two constructions, one which produces \emph{homology} classes in the moduli space and another which produces \emph{cohomology} classes. In this section we will derive expressions which determine what happens when we evaluate one construction upon the other. We will see that we can express the answer as the perturbative expansion of some functional integral over a finite space of fields.

Given a symplectic vector space $(V,\innprod)$, consider a Maurer-Cartan element
\[ x\in\mcs{\widehat{\Lambda}_{\gamma,\nu}[V]}. \]
In \cite{hamqme} such a structure was called a \emph{quantum $\ai$-structure}, since it was demonstrated in that paper that these structures are obtained from the deformation of an $\ai$-structure on $V$, which in turn is a kind of homotopy-invariant generalization of the notion of an associative algebra structure introduced by Stasheff in \cite{htopas}. In what follows we shall assume that $x$ is at least cubic in the grading on $\widehat{\Lambda}_{\gamma,\nu}[V]$ that was described in Section \ref{sec_dgla}. By Corollary \ref{cor_char} the quantum $\ai$-structure $x$ produces a family of \emph{homology} classes
\[ \ch(x) \in \prod_{\chi=-1}^{-\infty} \left[\bigoplus_{\begin{subarray}{c} g\geq 0, \ n\geq 1: \\ \chi=2-2g-n  \end{subarray}} H_\bullet(\lcmp/\mathbb{S}_n) \right]. \]

Suppose now we take a contractible differential graded Frobenius algebra $A$. In Section \ref{sec_ncbv} we described how this Frobenius algebra produces a family of \emph{cohomology} classes
\[ Q_A \in \prod_{\chi=-1}^{-\infty} \left[\bigoplus_{\begin{subarray}{c} g\geq 0, \ n\geq 1: \\ \chi=2-2g-n  \end{subarray}} H^\bullet(\lcmp/\mathbb{S}_n) \right] \]
with coefficients in $\gf[h]$.

We may take the cocycle $Q_A$ and evaluate it on the cycle $\ch(x)$. At first there may seem to be some uncertainty as to whether the result is well-defined or not, due to the implicit infinite summation that is involved. However, evaluating the component of the cocycle $Q_A$ with Euler characteristic $\chi=-n$ on the cycle $\ch(x)$ produces a homogenous polynomial in $h$ of order $n$. Hence evaluating $Q_A$ on $\ch(x)$ produces a well defined \emph{power series}
\[ Q_A[\ch(x)] \in \gf[[h]] \]
whose \emph{coefficients} represent the pairing of homology and cohomology classes, with real coefficients, in moduli spaces of curves of a \emph{fixed Euler characteristic}.

We would like to determine a more convenient expression for the power series produced by evaluating the cocycle $Q_A$ on $\ch(x)$. Consider the Maurer-Cartan element $x\in\mcs{\widehat{\Lambda}_{\gamma,\nu}[V]}$. Applying the map $\Phi_A$ defined by Figure \eqref{fig_dglamap}, which extends to the completed differential graded Lie algebras, to the Maurer-Cartan element $x$, produces a new Maurer-Cartan element
\[ y:=\Phi_A(x)\in\mcs{\widehat{\Lambda}_{\gamma,\nu}[V\otimes A]}. \]
This Maurer-Cartan element $y$ describes a quantum $\ai$-structure on $V\otimes A$. Let us comment briefly on how this structure is to be interpreted. It is of course well-known that there is a canonical way to make the tensor product of two algebras into an algebra. The same is true when we take the tensor product of an algebra and an $\ai$-algebra; there is a canonical way to give the tensor product the structure of an $\ai$-algebra (the same is not true for two $\ai$-algebras, see for instance \cite{markdiag} and \cite{umdiag}). The quantum $\ai$-structure $y$ is the result of tensoring the quantum $\ai$-structure $x$ on $V$ with the algebra $A$. Consider the potential
\[ \bar{y}:=\mathrm{N}(y)\in\gf[[h]]\cotimes\widehat{S}([V\otimes A]^*):=\gf[[h]]\cotimes\prod_{i=0}^\infty[(V\otimes A)^*]^{\otimes i}_{\mathbb{S}_i}. \]
We may use this potential to express the pairing of the two above classes.

\begin{theorem} \label{thm_pairing}
The power series arising from the evaluation of the cocycle $Q_A$ on the cycle $\ch(x)$ is computed by the following ``integral'',
\begin{equation} \label{eqn_partition}
Q_A[\ch(x)] = \ln\left[\frac{\int_{\tilde{L}} e^{\frac{1}{h}(\bar{y}-\sigma_d)}}{\int_{\tilde{L}}e^{-\frac{1}{h}\sigma_d}}\right],
\end{equation}
where the Lagrangian subspace $\tilde{L}:=V\otimes L$ is any Lagrangian subspace arising from a Hodge decomposition of $A$, as in Section \ref{sec_ncbv}.
\end{theorem}

\begin{rem}
Here the term integral is meant in a very loose sense, as the ``asymptotic expansion'', i.e. in terms of Feynman diagrams, of the integral about $h=0$. The term asymptotic expansion is itself also meant in a loose sense, since as the potential $y$ is a formal power series in $h$, there is no guarantee that even the integrand will converge for nonzero values of $h$.

To be more precise, and to avoid any ambiguity in the definition of the above expression, we may describe it as follows. Write the potential $\bar{y}$ as
\[ \bar{y} = \sum_{i=0,j=1}^\infty h^i\bar{y}_{ij}, \]
where $\bar{y}_{ij}$ is a homogeneous polynomial of order $j$; then
\[ Q_A[\ch(x)] = \ln\left[\sum_{n=0}^\infty\sum_{\begin{subarray}{c} i_1,\ldots,i_n = 0 \\ j_1,\ldots,j_n = 0 \\ \end{subarray}}^\infty \frac{h^{i_1+\cdots+i_n-n}}{n!}\frac{\int_{\tilde{L}}\bar{y}_{i_1 j_1}\cdots\bar{y}_{i_n j_n}e^{-\frac{1}{h}\sigma_d}}{\int_{\tilde{L}}e^{-\frac{1}{h}\sigma_d}}\right]. \]
Note that since all the terms in the above expression are integrals of polynomials with respect to a Gaussian measure, these terms are well-defined. As $\bar{y}_{01}=\bar{y}_{02}=0$, this series converges to a formal power series in $h$.
\end{rem}

\begin{proof}
Since the map $\phi:\widehat{C}_\bullet(\mathfrak{g})\to\widehat{C}_\bullet(\mathfrak{g})$ induced by a map $\phi:\mathfrak{g}\to\mathfrak{g}$ of differential graded Lie algebras is a map of commutative algebras, it follows that
\begin{displaymath}
\begin{split}
Q_A[\exp(x)] &= \frac{\int_{\tilde{L}} \mathrm{M}\mathrm{N}\Phi_A(e^x)\cdot e^{-\frac{1}{h}\sigma_d}}{\int_{\tilde{L}} e^{-\frac{1}{h}\sigma_d}} \\
&= \frac{\int_{\tilde{L}} \mathrm{M}(e^{\mathrm{N}\Phi_A(x)})\cdot e^{-\frac{1}{h}\sigma_d}}{\int_{\tilde{L}} e^{-\frac{1}{h}\sigma_d}} \\
&= \frac{\int_{\tilde{L}} e^{\frac{1}{h}(\bar{y}-\sigma_d)}}{\int_{\tilde{L}} e^{-\frac{1}{h}\sigma_d}}
\end{split}
\end{displaymath}
In this expression, the integral is computed in a purely formal manner by interchanging the integral with the infinite sum in the exponential, which arises due to the fact that we are now dealing with power series rather than polynomials. Hence, this expression yields the Feynman diagram expansion of this integral. Computing the expression $Q_A[\ch(x)]$ involves first projecting onto the \emph{primitive elements} of the Chevalley-Eilenberg homology. It is the primitive elements which correspond to \emph{connected} Riemann surfaces; or equivalently, in terms of the Feynman diagram expansion, \emph{connected} graphs. It is well-known, see e.g. \cite{manin}, that summing only over \emph{connected} diagrams in the Feynman diagram expansion leads to the \emph{natural logarithm} of the above expression.
\end{proof}

Let us take a moment to explain one of Kontsevich's original motivations for these constructions and their connection to Chern-Simons theory. The idea, loosely speaking, is that the partition function of Chern-Simons theory could be produced by such a pairing of these two constructions. The data needed for the construction which produces the cocycles is a differential graded Frobenius algebra $A$. In Chern-Simons theory, this data is supplied by taking the de Rham algebra of some compact space-time manifold $M$. The data needed to construct the cycles is a quantum $\ai$-algebra, which as we mentioned, will be obtained from some associative algebra. In Chern-Simons theory we are supposed to simply take our associative algebra to be a matrix algebra.

In pairing these two classes we arrive at an expression of the form \eqref{eqn_partition}. That is, we arrive at some functional integral whose space of fields over which we integrate is formed by taking the tensor product of a de Rham algebra and a matrix algebra; i.e. our space of fields consists of matrix-valued forms, or \emph{connections}. Since the kinetic term $\sigma_d$ is degenerate, in order to make sense of the functional integral it is necessary to fix a gauge, or rather to choose a Lagrangian subspace. This gauge is provided by placing a Riemannian metric on the space-time manifold $M$, under which the gauge-fixing condition becomes $d^*=0$. Furthermore, it follows from Theorem \ref{thm_cocycle} that the expression \eqref{eqn_partition} is not sensitive to this choice of a gauge condition. Rather, that is to say that it does not depend upon the choice of the Riemann metric on $M$, or in other words, it is a \emph{topological quantum field theory}.

There are, of course, serious issues with the above interpretation. One of the most significant of these is that throughout this paper we have necessarily assumed that our space of fields is finite-dimensional, whereas this is rarely the case in practice (although exceptions do exist), and leads to significant further complications such as the need for renormalization. Furthermore, it was necessary to assume that our differential graded algebras were contractible; another condition that is obviously difficult to accommodate in practice. Nonetheless, the machinery we have developed up to this point remains useful, and in the next section we will turn our attention to some examples. In particular, we will see how matrix integrals arise from our constructions.

\section{Matrix integrals and other computations} \label{sec_compute}

In this paper we have described two constructions producing classes in the moduli space of curves. In general, determining whether a given cohomology class is trivial or not is a difficult problem. One method is to evaluate that cohomology class upon a homology class, as we did in the previous section. If the number produced is nonzero, it follows that both our homology and cohomology classes are nontrivial. Hence, we may determine whether our constructions produce nontrivial classes by evaluating integral expressions of the form \eqref{eqn_partition}. This idea was originally proposed by Kontsevich in his seminal paper \cite{kontfeyn}, but until now it has not been carried out in any concrete example. In this section, we will work through the details in a simple example for the first time.

\subsection{Examples} \label{sec_examples}

Consider the parity reversion $V:=\Pi\gf$ of the ground field $\gf$, whose dual space is spanned by a single odd variable $t$ dual to $\Pi(1)$. The space $V$ has a canonical even symplectic form $\innprod_V:=t\otimes t$. In \cite[\S 6]{hamqme} it was shown that the element
\begin{equation} \label{eqn_exmqme}
x:=\sum_{i=1}^\infty a_i t^{2i+1}\in\mathfrak{h}_{\geq 3}[V]\subset\Lambda_{\gamma}[V]
\end{equation}
was a solution to the quantum master equation in $\Lambda_{\gamma}[V]$, for any choice of values of the coefficients $a_i\in\gf$. The origins of this example may be traced back to \cite{kontfeyn}. It follows from Corollary \ref{cor_char} that this Maurer-Cartan element gives rise to a family of homology classes $\ch(x)$ in the Kontsevich compactification of the moduli space with twisted coefficients.

The input for our second construction which produces cohomology classes will be taken from the paper \cite[\S 5]{dualfeyn}. This is a $1|1$-dimensional contractible differential graded algebra $\Xi$ generated by a unit $1$ in even degree and an odd generator $a$ satisfying $a^2=1$. The differential $d$ is determined by the equation $d(a)=1$ and its odd symmetric bilinear form is determined by $\langle a,1 \rangle = 1$. One may verify, by means of a simple calculation, that the map $\beta$ of Equation \eqref{eqn_beta} vanishes. It follows from Proposition \ref{prop_betavanish} that $\Xi$ gives rise to a family $Q_\Xi$ of cohomology classes on the Kontsevich compactification of the moduli space with twisted coefficients.

We may evaluate the pairing $Q_\Xi[\ch(x)]$ of these two constructions by the formulae of the preceding section. The Frobenius algebra $\Xi$ has a canonical Hodge decomposition specified by the equation $s(1)=a$, which determines an isotropic subspace $L=\im(s)$ of $\Xi$ and a Lagrangian subspace $\tilde{L}:=V\otimes L$ of $V\otimes \Xi$ of dimension $1|0$ generated by the element $\Pi(1)\otimes a$. The restriction of the tensor $t_n:\Xi^{\otimes n}\to\gf$ defined by Equation \eqref{eqn_product} to the subspace $L$ is described by the formulae
\begin{displaymath}
\begin{split}
t_n(a,\ldots,a)=1 & \quad n \text{ odd} \\
t_n(a,\ldots,a)=0 & \quad n \text{ even}
\end{split}
\end{displaymath}
From this we may conclude that the potential $\bar{y}:=\mathrm{N}\Phi_\Xi(x)$, when restricted to the Lagrangian subspace $\tilde{L}$, is given by
\[ \bar{y}|_{\tilde{L}}=\sum_{i=1}^\infty (-1)^i a_i \hat{t}^{2i+1}, \]
where $\hat{t}$ is the variable dual to the generator $\Pi(1)\otimes a$ of $\tilde{L}$. We may also calculate that the quadratic form $\sigma_d$ defined by \eqref{eqn_degform}, when restricted to $\tilde{L}$, is given by
\[ \sigma_d|_{\tilde{L}} = \hat{t}^2. \]
It follows from Theorem \ref{thm_pairing} that the pairing of $Q_\Xi$ and $\ch(x)$ is given by
\begin{equation} \label{eqn_pair}
Q_\Xi[\ch(x)] = \ln\left[\frac{\int_{\mathbb{R}} e^{\frac{1}{h}\left(\sum_{i=1}^\infty (-1)^i a_i \hat{t}^{2i+1} - \hat{t}^2\right)}d\hat{t}}{\int_{\mathbb{R}}e^{-\frac{1}{h}\hat{t}^2}d\hat{t}}\right].
\end{equation}

Let us compute this expression in the simple case where one $a_i$ is nonzero, say $a_i=(-1)^i$, and all of the other coefficients $a_j:j\neq i$ are zero. Let us denote the corresponding solution to the quantum master equation by $\tilde{x}_i$. In this case, a completely standard calculation involving standard formulae for the integral of a polynomial function against a Gaussian measure implies that
\[ Q_\Xi[\ch(\tilde{x}_i)] = \ln\left[\sum_{n=0}^\infty \frac{[2n(2i+1)]!}{2^{n(2i+1)}[n(2i+1)]![2n]!} h^{n(2i-1)}\right]. \]

It is clear from this calculation that both the cohomology class coming from the differential graded algebra $\Xi$ and the homology classes coming from the solution $\tilde{x}_i$ to the quantum master equation are nontrivial. In fact, one may easily show that the homology class
\[ \ch(\tilde{x}_i) \in \prod_{\chi=-1}^{-\infty} \left[\bigoplus_{\begin{subarray}{c} g\geq 0, \ n\geq 1: \\ \chi=2-2g-n  \end{subarray}} H_\bullet(\lcmp/\mathbb{S}_n) \right]. \]
is concentrated in degrees $-\frac{2i+1}{2i-1}\chi-1$, where $\chi$ is the corresponding Euler characteristic of the relevant component of the moduli space. In particular we see that the element $\tilde{x}_1$ produces top dimensional classes, and that the other solutions produce classes in higher codimensions. Hence we see that the cohomology class $Q_\Xi$ has nontrivial components in dimensions and codimensions other than zero.

That the homology classes $\ch(\tilde{x}_i)$ are nontrivial is to be expected, as it was shown independently in \cite{igusa} and \cite{mondello} that the classes $\ch(x)$ for a general choice of coefficients generate the kappa classes in the moduli space. The cohomology class $Q_\Xi$ is new however, and with a little work, we may refine slightly our knowledge of its nontrivial components.

\begin{prop}
The cohomology class
\[ Q_\Xi \in \prod_{\chi=-1}^{-\infty} \left[\bigoplus_{\begin{subarray}{c} g\geq 0, \ n\geq 1: \\ \chi=2-2g-n  \end{subarray}} H_\bullet([\kcmp\times\Delta_n^\circ]^\infty/\mathbb{S}_n)\right] \]
has a nontrivial component in every Euler characteristic.
\end{prop}

\begin{proof}
We will explain how one may choose the parameters $a_i\in\gf$ of \eqref{eqn_exmqme} so that when we compute the expression $Q_A[\ch(x)]$ using Equation \eqref{eqn_partition}, we obtain a power series in $h$ all of whose coefficients are nontrivial. This will suffice to establish the proposition.

In computing this expression using \eqref{eqn_partition}, we may use the Taylor series expansion
\begin{equation} \label{eqn_natlog}
\ln(1+x) = \sum_{n=1}^\infty \frac{(-1)^{n+1}}{n} x^n.
\end{equation}
We begin by choosing $a_1=-1$ and all other coefficients to be zero. Denote by $x_1$ the corresponding solution to the quantum master equation given by Equation \eqref{eqn_exmqme}. Using \eqref{eqn_natlog} we may compute that
\[ Q_\Xi[\ch(x_1)]= \frac{15}{2}h+O(h^2). \]

Denote this power series by $z_1(h)$ and define $p_1(h):=e^{z_1(h)}-1$. If the coefficient of $h^2$ in $z_1(h)$ is nonzero, then we may leave $a_2=0$. Otherwise, we leave $a_1=-1$, choose $a_2=1$ and set all other coefficients to zero. In this latter case, denote the new corresponding solution to the quantum master equation by $x_2$; again, we may compute that
\begin{displaymath}
\begin{split}
Q_\Xi[\ch(x_2)] & = \ln\left[1+\frac{105}{2}h^2 + p_1(h) + O(h^3)\right] \\
& = z_1(h)+\frac{105}{2}h^2+O(h^3) \\
& = \frac{15}{2}h+\frac{105}{2}h^2 + O(h^3).
\end{split}
\end{displaymath}
The result is a power series $z_2(h)$ whose first two coefficients are nonzero.

We may proceed inductively, defining the $a_i$ as we go, by setting $a_i=(-1)^i$ if the coefficient of $h^i$ in $z_{i-1}(h)$ is zero, and setting $a_i=0$ if it is not. The result is a sequence of solutions $\{x_n\}_{n\in\mathbf{N}}$ to the quantum master equation, which converge in the $t$-adic topology to a solution $x_\infty$ such that
\[Q_\Xi[\ch(x_\infty)]= \lim_{n\to\infty} Q_\Xi[\ch(x_n)] = \lim_{n\to\infty} z_n(h) \]
is a power series in $h$ with no nonzero coefficients. It follows that the cohomology class $Q_\Xi$ must have a nontrivial component in every Euler characteristic.
\end{proof}

\subsection{Matrix integrals}

Given any Frobenius algebra $A$, we may tensor this algebra with the contractible $1|1$-dimensional differential graded algebra $\Xi$ defined in the previous section, to produce a new contractible differential graded Frobenius algebra $E:=A\otimes\Xi$. Since $\Xi$ satisfies the hypothesis of Proposition \ref{prop_betavanish}, it follows simply that $E$ does too. Consequently, we arrive easily at the following theorem.

\begin{theorem}
Any Frobenius algebra $A$ produces a cohomology class
\[ Q_{A\otimes\Xi} \in \prod_{\chi=-1}^{-\infty} \left[\bigoplus_{\begin{subarray}{c} g\geq 0, \ n\geq 1: \\ \chi=2-2g-n  \end{subarray}} H_\bullet([\kcmp\times\Delta_n^\circ]^\infty/\mathbb{S}_n)\right] \]
in the Kontsevich compactification of the moduli space of Riemann surfaces.
\end{theorem}
\noproof

Now suppose that we take a matrix algebra $M_n(\gf)$ as our Frobenius algebra. Then $E:=M_n(\gf)\otimes \Xi$ has an odd symmetric nondegenerate bilinear form, and a canonical isotropic subspace isomorphic to $M_n(\gf)\otimes\Pi\gf$. In a similar manner to Equation \eqref{eqn_pair}, one may compute that
\[ Q_E[\ch(x)] = \ln\left[\frac{\int_{M_n(\gf)} e^{\frac{1}{h}\left(\sum_{i=1}^\infty (-1)^ia_i\tr(A^{2i+1})-\tr(A^2)\right)} dA}{\int_{M_n(\gf)} e^{-\frac{1}{h}\tr(A^2)}dA}\right] \]
where $x$ is the solution \eqref{eqn_exmqme} to the quantum master equation. In fact, one could write down a similar expression for any Frobenius algebra, not just matrix algebras. These matrix integrals bear remarkable similarity to those integrals considered by Kontsevich in his proof of Witten's conjectures \cite{kontairy}.

\end{document}